\documentclass[gdweb]{geradwp}

\PassOptionsToPackage{hyphens}{url}

\usepackage[french,english]{babel}
\usepackage{hyperref}
\usepackage{tikz}
\usepackage{todonotes}
\usetikzlibrary{arrows.meta, positioning}
\usepackage{tikz}
\usepackage{amsmath}
\usepackage{amssymb}
\usepackage{amsthm}  
\usepackage{url}
\usepackage{placeins}

\usepackage{hyperref}
\usetikzlibrary{positioning, shapes.geometric, arrows.meta, decorations.pathreplacing}
\usepackage[sort,numbers]{natbib}
\usepackage{amsthm}
\theoremstyle{plain}
\newtheorem{corollary}{\sffamily Corollary}

\newtheorem{proposition}{\sffamily Proposition}

\theoremstyle{definition}

\theoremstyle{remark}
\newtheorem{remark}{\sffamily Remark}

\newcolumntype{C}[1]{>{\centering\arraybackslash}p{#1}}

\usepackage{tikz}
\usepackage{float}
\usepackage{rotating}
\usepackage{adjustbox}
\usepackage{makecell}
\usepackage{multicol}
\usepackage{amsmath}
\usepackage{amssymb}
\usetikzlibrary{positioning,chains,fit,shapes,calc,arrows.meta,                        
	  backgrounds,                        
	  ext.paths.ortho,                    
	  ext.positioning-plus,               
	  }
\usepackage{footmisc}
\usepackage{pgfplots}
\pgfplotsset{compat=1.18}
\usetikzlibrary{spy}
\usepackage{newtxtext,newtxmath}
\usepackage{svg}
\usepackage{fontawesome}
\usepackage{amsthm}
\usepackage{mathtools}
\usepackage{enumitem}
\usepackage{soul}
\usepackage{graphicx}
\usepackage{longtable}
\usepackage{algorithm}
\usepackage{algorithmic}
\usepackage{graphicx}   
\usepackage{tikz}
\usetikzlibrary{arrows.meta}
\usetikzlibrary{shapes.geometric, arrows}
\soulregister\cite7
\soulregister\ref7
\soulregister\pageref7
\tikzstyle{startstop} = [rectangle, rounded corners, minimum width=3cm, minimum height=1.2cm, text centered, draw=black, fill=gray!50]
\tikzstyle{process} = [rectangle, minimum width=4cm, minimum height=1.2cm, text centered, draw=black, fill=white]
\tikzstyle{repeated} = [rectangle, minimum width=4cm, minimum height=1.2cm, text centered, draw=black, fill=gray!20, dashed]
\tikzstyle{decision} = [diamond, minimum width=3cm, minimum height=1.5cm, text centered, draw=black, fill=gray!30]
\tikzstyle{arrow} = [thick,->,>=stealth]

\GDcoverpagewhitespace{6.8cm}
\graphicspath{{Figures/}} 
\hypersetup{colorlinks,%
citecolor={blue}, 
urlcolor={blue},
linkcolor={blue},
breaklinks={true}
}

\makeatletter
\ifthenelse{\isundefined{\ALG@name}}{}%
{%
\renewcommand{\ALG@name}{\sffamily\footnotesize Algorithm}
}
\makeatother
\ifthenelse{\isundefined{\SetAlCapNameFnt}}{}%
{%
\SetAlCapNameFnt{\footnotesize}
\SetAlCapFnt{\sffamily\footnotesize}
}

\GDauthorsShort{A.~Abdellaoui, L.~Benabbou, I.~El~Hallaoui}
\GDauthorsCopyright{Abdellaoui, El~Hallaoui, Benabbou}
\GDpostpubcitation{Hamel, Benoit, Karine H\'ebert (2024). 
``Un exemple de citation'', \emph{Journal of Journals}, 
vol. X issue Y, p. n-m}{https://www.gerad.ca/fr}
\GDsupplementname{Internet Appendix}
\GDrevised{Mai}{May}{2024}

\title{\textbf{Real-World, Large Scale, Multi-Period Log Truck 
Routing and Scheduling : Application to Canadian Forestry}}

\author{
Abdelhakim Abdellaoui$^{1,3,}$\thanks{Corresponding author: abdelhakim.abdellaoui@polymtl.ca} \quad
Loubna Benabbou$^{2}$ \quad
Issmail El Hallaoui$^{1}$ 
 \\
 Fran\c{c}ois Aub\'e$^{3}$\quad
Mouloud Amazouz$^{3}$
\\[1em]
{\small \textit{$^{1}$Department of Mathematics and Industrial Engineering, Polytechnique Montr\'eal, succ. Centre-ville,Montr\'eal, Qu\'ebec, H3C 3A7, Canada}} \\
{\small \textit{$^{2}$Department of Management Science, Université du Québec à Rimouski (UQAR), Levis, QC G6V 0A6, Qu\'ebec, Canada}}\\
{\small\textit{ $^{3}$Natural Resources Canada, CanmetENERGY, 1615 Lionel-Boulet Blvd, P.O. Box 4800, Varennes, Qu\'ebec, J3X 1P7, Canada}} 
}
\begin{document}
\date{}
\maketitle



\begin{GDabstract}{Abstract}
This paper addresses the multi-period log-truck routing and 
scheduling problem ($\mathcal{LTRSP}$), a key operational 
activity in the forestry industry, where transportation 
accounts for more than one-third of total operational costs. 
The Canadian forestry sector faces significant logistical 
difficulties driven by vast geographic distances, seasonal 
variability, volatile markets, and environmental considerations. 
Our research tackles a long-standing open question in the 
forestry operations literature, namely the absence of an exact 
and scalable formulation of the forestry vehicle routing problem 
integrating the full range of operational constraints 
\cite{ronnqvist2015operations}: \textit{How can we model and 
solve an exact formulation of the forestry VRP problem?} 
In response, we analyze business rules specific to the forestry 
sector to construct a routing network reflecting real operational 
practices, and then propose a comprehensive improved 
mixed-integer linear programming (MILP) formulation 
incorporating all known forestry operational constraints, 
with a detailed justification of key modeling choices such as 
time discretization, along with a decomposition-based solution 
methodology tailored for large-scale, multi-period industrial 
instances. To address the inherent combinatorial complexity, 
we combine state-of-the-art solvers with a metaheuristic 
decomposition strategy based on \textit{Relax\&Fix} and 
\textit{Fix\&Optimize}. Computational experiments on historical 
data from a Canadian forest company demonstrate near-optimal 
results within practical computation times, with significant 
financial gains and reduced greenhouse gas emissions.
\end{GDabstract}
\vspace{-2cm}
\section{Introduction}
In Canada, the forestry sector constitutes one of the primary drivers of economic activity. According to Canada's Forests Annual Report 2023 
\cite{NRCan2022}, nominal GDP from forest-related activities reached 33.4\$ billion in 2022, with logging operations contributing 
\$5.9 billion. This sector not only generates economic value but also supports employment and sustainable development \cite{NRCan2022}. Such significance underscores the pressing need for efficient transportation, particularly as the industry faces evolving challenges. Log transportation in Canada is subject to a considerable number  of challenges. With nearly 226 million hectares of managed forest  land \cite{NRCan2022}, the geographical scale of harvesting operations  requires an extensive transportation infrastructure. Road accessibility  is further affected by climate conditions, as winter and spring  periods periodically restrict or degrade transportation routes. At the same time, the sector faces growing pressure to meet  environmental sustainability requirements, adding further constraints  to an already complex operational setting.

In front of these challenges, Canadian forest companies strongly prioritize cost reduction in order to maintain adequate profit 
margins and remain competitive in an increasingly volatile market \cite{StatCan2023}.  Log transportation makes up to 40 \% of the expenses for a typical forest company \cite{audy2013virtual}. Consequently, optimizing transportation costs can be considered as a potential field to bring economic added value. In addition, reducing transportation costs by minimizing the total distance traveled is inherently linked to a reduction in greenhouse gas emissions, supporting environmental sustainability. The ultimate goal of optimizing log-truck routing and scheduling fits within this perspective, in addition to ensuring operational efficiency. Furthermore, truck-based logistics is the primary mode of transport adopted by forestry companies \cite{StatCan2023}. To address all these considerations, the forestry sector is increasingly looking at digitization and optimization of their logistic operations through advanced techniques. The current business model involves subcontracting harvesting and wood transportation to contractors. Studies have shown that a more centralized decision-making system improves operational efficiency \cite{damavsevivcius2024digital}. In this context, some companies are turning to advanced operations research methods to design optimal and centralized transportation planning solvers. Moreover, the literature highlights that numerous open questions 
within forestry logistics remain to be addressed by the operations research community \cite{ronnqvist2015operations}. These primarily involve developing and solving an exact model that incorporates various critical aspects, such as managing queuing at loading and unloading locations, synchronizing trucks without self-loading capabilities with loaders at supply and demand points, and handling a combination of full and partial truckloads \cite{yahiaoui2024mathheuristic}. Considering the aforementioned challenges, this paper aims to design an end-to-end framework to tackle a rich problem that integrates key forestry-related aspects, focusing on a large-scale, multi-period log-truck routing and scheduling problem. Following comprehensive discussions with Canadian forest companies,  we approached this problem from their perspective. Our goal is to  provide a practical and effective solution that facilitates transportation planning, a process that currently relies heavily  on manual and decentralized decision-making. The complete framework we propose begins with an improved mathematical model formulation grounded in standard practices prevalent in the Canadian forest industry. We then implement a generic methodology to generate a valid time-space routing network for any real-world instance. In this context, a routing network is defined as the set of all nodes and edges, while considering some conventional constraints such as product assignments to forest sites and mills, as well as vehicle assignments to home bases. In the next phase, the routing network undergoes a processing methodology according to specific rules to minimize its size and, consequently, reduce the combinatorial space. Finally, we design tailored solution approaches to generate optimal truck routes for a defined planning horizon, ensuring that all business constraints are satisfied. 

The present article has a fourfold contribution: (1) we formulate a mathematical model tailored to the large-scale, multi-period log-truck routing and scheduling problem, incorporating the common 
constraints found in the Canadian forest industry; (2) we discuss the characteristics and sources of complexity of the model and provide mathematical insights into time discretization; (3) we propose 
an efficient solving approach based on decomposition methods, and provide the intuition behind its use; and (4) we evaluate the performance of our proposed approach on several real-world instances and assess its financial, environmental, and operational impact. Through this research, we seek to provide insights that can help the Canadian forestry sector maintain its competitive edge while adapting to the evolving challenges of the 21st century. The findings may have implications not only for Canada but also for other countries with significant forestry industries facing similar logistical challenges. The following section presents an overview of the relevant literature.
\section{Literature review}
\label{lit-rev}
The optimization of log-truck routing and scheduling in forestry has been the subject of many research due to its significant operational, economic, and environmental implications. This review focuses on the evolution of approaches that address this complex problem in terms of formulation and solution methods.

The foundational work in $\mathcal{LTRSP}$ emerged in the 1980s, emphasizing the attributes that distinguish these problems from general VRPs, such as full and partial truckloads, the need for multi-depot planning, and synchronization constraints at supply and demand sites \cite{audy2023review}. The research then expanded significantly, incorporating more real-world constraints \cite{ronnqvist2003optimization}. Since the 1990s, many researchers have proposed various formulations and solving techniques in this field. Weintraub et al. \cite{weintraub1996simulation} developed a heuristic-based decision-making tool known as ASICAM, which integrated simulation with heuristic rules to create truck routes and schedules while managing queuing at supply and demand locations. This system was particularly effective in minimizing wait times, but was limited by its heuristic nature, which lacked mechanisms for iterative solution refinement. The work of Rönnqvist and al. \cite{ronnqvist2015operations} emphasized critical challenges in $\mathcal{LTRSP}$ modeling, such as handling full and partial truckloads, synchronization between loading equipment and vehicle schedules, and real-time adjustments.

Further advancements came through network-based models and exact solution methods. Shen et al. \cite{shen1978log} presented a capacitated linear programming network flow model, focusing on the optimal distribution of truck routes to minimize transportation costs under time window constraints. Robinson et al. \cite{robinson1994tour} expanded on this by incorporating heterogeneous vehicle fleets and multi-origin, multi-destination problems, showing the flexibility of network flow models in handling complex transportation networks.

The use of decomposition methods gained traction as a way to tackle the inherent complexity of $\mathcal{LTRSP}$. El Hachemi et al. \cite{el_hachemi2013decomposition} introduced a two-phase decomposition approach for weekly tactical scheduling, integrating mixed-integer programming (MIP) for strategic allocation and constraint programming (CP) for detailed daily routing and queuing time optimization. This combination allows efficient synchronization of truck and loader operations, addressing a major limitation of earlier heuristic-based methods.

Later, the limitations of exact methods to solve large-scale $\mathcal{LTRSP}$ led to the development of heuristic and hybrid approaches. Palmgren et al. \cite{palmgren2004column} utilized a column generation method to solve a combined timetabling and routing problem. This method effectively embedded supply and demand constraints, but faced computational challenges during the pricing phase, particularly when addressing resource-constrained shortest path problems. Yahiaoui et al.\cite{yahiaoui2024mathheuristic} introduced a metaheuristic approach for the $\mathcal{LTRSP}$ with queuing considerations ($\mathcal{LTRSP}$Q). This two-phase method involved generating an initial solution pool using a rolling horizon heuristic, followed by a column generation-based procedure to refine solutions by minimizing queuing times. The innovative aspect of this approach was its adaptive time slot allocation, which optimized the scheduling of loading and unloading operations. Dynamic and stochastic models have emerged to better represent real-world uncertainties, such as changes in log availability, delivery demands, and vehicle breakdowns. Rönnqvist at al.\cite{ronnqvist1995real} proposed a method for real-time truck dispatching that incorporated dynamic data updates. Similarly, Marques et al. \cite{marques2014hybrid} combined optimization and simulation to manage the stochastic elements of truck arrivals, delays, and queuing, effectively bridging the gap between strategic planning and operational execution.

Despite these advances, significant challenges remain unaddressed. Computational efficiency continues to be a barrier, especially when large-scale problems are addressed. Moreover, existing solutions lack a complete integration of all forestry challenges within a single solution. Key elements such as multi-horizon planning, synchronization, adherence to the most important business constraints, and flexibility for rescheduling in response to unforeseen events are often missing, as underscored in \cite{audy2023planning}. This table sums up the contributions the most significant works in this field:
\begin{table}[h!]
    \centering
    \renewcommand{\arraystretch}{1.2}
    \setlength{\tabcolsep}{5pt}
    \begin{adjustbox}{max width=\textwidth}
    \begin{tabular}{lcccccccccccc}
        \toprule
        \textbf{Paper} & \textbf{Horizon planning} & \textbf{Multi-Period} & \textbf{Hetero.Fleet} & \textbf{Time Windows} & \textbf{Multi-products} & \textbf{Queuing} & \textbf{Incomplete trip} & \textbf{Method} & \textbf{Partial load} & \textbf{Priority rules} & \textbf{Self loading trucks} & \textbf{Pairing} \\
        \midrule
        \textit{Weintraub et al. \cite{weintraub1996simulation} (1996)} & daily & & $\checkmark$ & & $\checkmark$ & $\checkmark$ & & Heuristic & $\checkmark$ & $\checkmark$ & & \\
        \textit{Palmgren et al. \cite{palmgren2004column} (2004)} & daily & & & $\checkmark$ & $\checkmark$ & & & Exact & $\checkmark$ & & & \\
        \textit{Rey et al. \cite{rey2009column} (2009)} & daily & & $\checkmark$ & $\checkmark$ & $\checkmark$ & $\checkmark$ & & Exact & & & & \\
        \textit{El Hachemi et al. \cite{el2011hybrid} (2011)} & daily & & $\checkmark$ & $\checkmark$ & $\checkmark$ & $\checkmark$ & & Hybrid & $\checkmark$ & & & \\
        \textit{El Hachemi et al. \cite{el_hachemi2013decomposition} (2013)} & daily & & $\checkmark$ & $\checkmark$ & $\checkmark$ & $\checkmark$ & & Exact & & & & \\
        \textit{Haridass et al. \cite{haridass2014scheduling} (2014)} & daily & & $\checkmark$ & $\checkmark$ & & $\checkmark$ & & Exact & $\checkmark$ & & & \\
        \textit{Rix et al. \cite{rix2015column} (2014)} & Annualy & $\checkmark$ & $\checkmark$ & & $\checkmark$ & $\checkmark$ & $\checkmark$ & Exact & & & $\checkmark$ & \\
        \textit{Acuna et al. \cite{acuna2017timber} (2017)} & weekly & $\checkmark$ & $\checkmark$ & & $\checkmark$ & $\checkmark$ & $\checkmark$ & Exact & $\checkmark$ & & & \\
        \textit{Bordon et al. \cite{bordon2020mixed} (2020)} & weekly & $\checkmark$ & $\checkmark$ & & $\checkmark$ & & $\checkmark$ & Exact & $\checkmark$ & & & \\
        \textit{Melchiori et al. \cite{melchiori2022mathematical} (2022)} & weekly & $\checkmark$ & $\checkmark$ & & $\checkmark$ & $\checkmark$ & $\checkmark$ & Exact & $\checkmark$ & $\checkmark$ & & \\
        \textit{Ghotb et al. \cite{ghotb2024optimization} (2022)} & daily & & $\checkmark$ & $\checkmark$ & & & $\checkmark$ & Metaheuristic & $\checkmark$ & & & \\
        \textit{Our paper} \cite{ghotb2024optimization} & weekly & $\checkmark$ & $\checkmark$ & $\checkmark$ & $\checkmark$ & $\checkmark$ & $\checkmark$ & Exact/Metaheuristic & $\checkmark$ & $\checkmark$ & $\checkmark$ & $\checkmark$ \\
        \bottomrule
    \end{tabular}
    \end{adjustbox}
    \caption{Characteristics summary of past studies vs this paper.}
    \label{tab:summary}
\end{table}

These gaps were highlighted by business stakeholders who have tested various approaches in the real-world context. The existing solutions fail to capture the full complexity of real-world forestry operations. Consequently, both industry experience and academic literature indicate a strong trend toward integrated approaches that address all forestry challenges, taking into account the deployment of algorithmic solutions using user-friendly interfaces.

The primary objective of this paper is to outline the key elements of such an integrated approach, aimed at tackling the common challenges in forestry transportation planning. Specifically, this paper presents an end-to-end framework for truck-log routing and scheduling. Future contributions will focus on the development of a rescheduling layer to adapt to unforeseen events.

\section{Problem description}
Given our intention in this paper to present a general approach to formulate and solve $\mathcal{LTRSP}$, we divide this section into two parts. The first part introduces the used notation throughout the paper. Then, the second one sheds light on the Canadian forestry context and the problem statement.
\label{prob-des}

\subsection{Mathematical notation}
We first provide a detailed overview of the mathematical notation used throughout this paper. The sets, indices, parameters and decision variables
are introduced in Table \ref{tab1}.
\begin{table}[H]
\centering
\caption{Model sets and parameters}
\renewcommand{\arraystretch}{1.3}
\begin{tabular}{ll}
\toprule
\textbf{Symbol} & \textbf{Definition} \\
\midrule
$F$ & set of forest sites \\
$M$ & set of mills \\
$V$ & set of trucks \\
$P$ & set of wood products \\
$I$ & set of time intervals \\
$N$ & set of nodes representing mills, forest sites, and home bases \\
$\mathcal{HB}_v$ & set of eligible home bases for truck $v$ \\
$T$ & set of days in the planning horizon \\
$A$ & set of arcs connecting pairs of nodes across time intervals \\
$A^+(n)$ & set of arcs leaving node $n$ \\
$A^-(n)$ & set of arcs entering node $n$ \\
$A_{f}^{mp}$ & set of loaded arcs from forest site $f$ to mill $m$ for product $p$ \\
$Source_v$ & origin node for truck $v$ (departure from home base) \\
$Sink_v$ & destination node for truck $v$ (return to home base) \\
$A_{fi}^{L}$ & loading arcs at forest site $f$ at time interval $i$ \\
$A_{mi}^{U}$ & unloading arcs at mill $m$ at time interval $i$ \\
$c_{a}^{h}$ & hauling cost for arc $a$ traversed by vehicle $v$ \\
$c_{a}^{w}$ & waiting cost for arc $a$ traversed by vehicle $v$ \\
$c_{mp}$ & penalty cost per unit of unmet demand for product $p$ at mill $m$ \\
$\tau_{a}^{+}$ & ending time interval of arc $a$ \\
$\tau_{a}^{-}$ & starting time interval of arc $a$ \\
$t_f$ & last day of horizon $T$ \\
$d_{mp}$ & weekly demand for product $p$ at mill $m$, expressed in GMT \\
$q_{v}$ & payload capacity of vehicle $v$, expressed in GMT \\
$K_{v}$ & maximum number of trips allowed for vehicle $v$ \\
$\beta_{v}$ & indicator equal to 1 if vehicle $v$ lacks self-loading capability \\
$s_{fp}$ & available supply of product $p$ at forest site $f$, expressed in GMT \\
$S_{t}$ & service duration, corresponding to loading time at forest sites 
          and unloading time at mills \\
$\Lambda_{mi}$ & loader capacity at mill $m$ during time interval $i$ \\
$\Lambda_{fi}$ & loader capacity at forest site $f$ during time interval $i$ \\
$T_{a}^{v}$ & travel duration for arc $a$ by vehicle $v$ \\
$T^o_{n}, T^c_{n}$ & operating window (opening and closing times) at node $n$ \\
\textit{BigM} & sufficiently large constant \\
\bottomrule
\end{tabular}
\label{tab1}
\end{table}
\subsection{Canadian forestry context and problem statement}
 The study is tailored to the general context of the wood industry, with a particular focus on the Canadian forestry sector. In collaboration with a Canadian forestry company, we model and solve a vehicle routing problem formulation that captures the key operational constraints of $\mathcal{LTRSP}$. Our objective is to incorporate, to the greatest extent possible, the constraints faced by a broad range of forestry operations, thereby enabling the development of a generic solution approach. This comprehensive modeling framework differentiates our work from existing studies in the literature.

In the context of $\mathcal{LTRSP}$,  a route consists of an ordered sequence of loaded and empty trips, each associated with defined departure and arrival times. Each route must begin and end at a corresponding home base, typically the truck operator’s residence or yard. Specifically, three types of empty trips are considered: a truck may travel from its home base to a forest site, from a mill to a forest site to pick up new products, or from a mill directly back to its home base at the end of the day. In addition, only one type of loaded trip exists, which is from a forest site to a mill. In rare instances, drivers may transport a loaded truck from a forest site back to their home base to spend the night, continuing the trip to the mill the following morning. Figure  ~\ref{fig:routes} represent different possible combinations of loaded and empty trips.

\tikzset{
  loaded/.style={->, very thick, draw=violet!70!black},
  empty/.style={->, very thick, draw=orange!85!black},
  startend/.style={->, thick, dashed, draw=blue!70!black}
}

\begin{figure}[h!]
\centering

\begin{minipage}[b]{0.45\textwidth}
\centering
\textbf{(a)}\\[0.3em]
\begin{tikzpicture}[scale=0.35, every node/.style={scale=0.4}]
    \node[draw,circle,fill=brown,inner sep=2pt] (Source) at (-1,-2.5) {}; 
    \node[anchor=east] at (Source) {Home base $HB$};
    \node[draw,circle,fill=brown,inner sep=2pt] (Sink) at (13,-2.5) {};
    \node[anchor=west] at (Sink) {Home base $HB$};

    \foreach \y in {0,...,13} {
        \node[draw,circle,fill=brown,inner sep=1.5pt] (F\y) at (\y,-4) {};
        \node[draw,circle,fill=brown,inner sep=1.5pt] (M\y) at (\y,-6) {};
    }

    \node[anchor=east] at (F0) {Forest site $F$};
    \node[anchor=east] at (M0) {Mill $M$};

    \draw[startend] (Source) -- (F0);
    \draw[loaded]   (F1) -- (M3);
    \draw[empty]    (M4) -- (F7);
    \draw[loaded]   (F8) -- (M10);
    \draw[startend] (M11) -- (Sink);

    \draw[->, dashed, very thick, draw=red!75!black,
          shorten <=2pt, shorten >=2pt] (F0)--(F1);
    \draw[->, dashed, very thick, draw=yellow!85!black,
          shorten <=2pt, shorten >=2pt] (M3)--(M4);
    \draw[->, dashed, very thick, draw=red!75!black,
          shorten <=2pt, shorten >=2pt] (F7)--(F8);
    \draw[->, dashed, very thick, draw=yellow!85!black,
          shorten <=2pt, shorten >=2pt] (M10)--(M11);

    \node at ($(F0)!0.5!(F1)+(0,0.85)$)
      {\includegraphics[width=1.5cm]{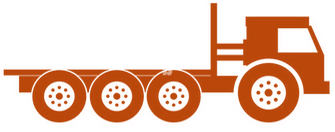}};
    \node at ($(M3)!0.5!(M4)+(0,-0.85)$)
      {\includegraphics[width=1.5cm]{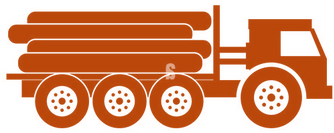}};
    \node at ($(F7)!0.5!(F8)+(0,0.85)$)
      {\includegraphics[width=1.5cm]{loading.png}};
    \node at ($(M10)!0.5!(M11)+(0,-0.85)$)
      {\includegraphics[width=1.5cm]{unloading.png}};
\end{tikzpicture}
\end{minipage}
\hfill
\begin{minipage}[b]{0.45\textwidth}
\centering
\textbf{(b)}\\[0.3em]
\begin{tikzpicture}[scale=0.35, every node/.style={scale=0.4}]
    \node[draw,circle,fill=brown,inner sep=2pt] (Source) at (-1,-2.5) {};
    \node[anchor=east] at (Source) {Home base $HB$};
    \node[draw,circle,fill=brown,inner sep=2pt] (Sink) at (13,-2.5) {};
    \node[anchor=west] at (Sink) {Home base $HB$};

    \foreach \y in {0,...,13} {
        \node[draw,circle,fill=brown,inner sep=1.5pt] (F\y) at (\y,-4) {};
        \node[draw,circle,fill=brown,inner sep=1.5pt] (M\y) at (\y,-6) {};
        \node[draw,circle,fill=brown,inner sep=1.5pt] (Ftwo\y) at (\y,-8) {};
    }

    \node[anchor=east] at (F0) {Forest site $F_1$};
    \node[anchor=east] at (M0) {Mill $M$};
    \node[anchor=east] at (Ftwo0) {Forest site $F_2$};

    \draw[startend] (Source) -- (F0);
    \draw[loaded]   (F1) -- (M3);
    \draw[empty]    (M4) -- (Ftwo7);
    \draw[loaded]   (Ftwo8) -- (M10);
    \draw[startend] (M11) -- (Sink);

    \draw[->, dashed, very thick, draw=red!75!black,
          shorten <=2pt, shorten >=2pt] (F0)--(F1);
    \draw[->, dashed, very thick, draw=yellow!85!black,
          shorten <=2pt, shorten >=2pt] (M3)--(M4);
    \draw[->, dashed, very thick, draw=red!75!black,
          shorten <=2pt, shorten >=2pt] (Ftwo7)--(Ftwo8);
    \draw[->, dashed, very thick, draw=yellow!85!black,
          shorten <=2pt, shorten >=2pt] (M10)--(M11);

    \node at ($(F0)!0.5!(F1)+(0,0.85)$)
      {\includegraphics[width=1.5cm]{loading.png}};
    \node at ($(M3)!0.5!(M4)+(0,-0.85)$)
      {\includegraphics[width=1.5cm]{unloading.png}};
    \node at ($(Ftwo7)!0.5!(Ftwo8)+(0,0.85)$)
      {\includegraphics[width=1.5cm]{loading.png}};
    \node at ($(M10)!0.5!(M11)+(0,-0.85)$)
      {\includegraphics[width=1.5cm]{unloading.png}};
\end{tikzpicture}
\end{minipage}

\par\bigskip

\begin{minipage}[b]{0.5\textwidth}
\centering
\textbf{(c)}\\[0.3em]
\begin{tikzpicture}[scale=0.4, every node/.style={scale=0.4}]
    \node[draw,circle,fill=brown,inner sep=2pt] (Source) at (-1,-2.5) {};
    \node[anchor=east] at (Source) {Home base $HB$};
    \node[draw,circle,fill=brown,inner sep=2pt] (Sink) at (13,-2.5) {};
    \node[anchor=west] at (Sink) {Home base $HB$};

    \foreach \y in {0,...,13} {
        \node[draw,circle,fill=brown,inner sep=1.5pt] (F\y) at (\y,-4) {};
        \node[draw,circle,fill=brown,inner sep=1.5pt] (M\y) at (\y,-6) {};
        \node[draw,circle,fill=brown,inner sep=1.5pt] (Ftwo\y) at (\y,-8) {};
        \node[draw,circle,fill=brown,inner sep=1.5pt] (Mtwo\y) at (\y,-10) {};
    }

    \node[anchor=east] at (F0) {Forest site $F_1$};
    \node[anchor=east] at (M0) {Mill $M_1$};
    \node[anchor=east] at (Ftwo0) {Forest site $F_2$};
    \node[anchor=east] at (Mtwo0) {Mill $M_2$};

    \draw[startend] (Source) -- (F0);
    \draw[loaded]   (F1) -- (M3);
    \draw[empty]    (M4) -- (Ftwo7);
    \draw[loaded]   (Ftwo8) -- (Mtwo10);
    \draw[startend] (Mtwo11) -- (Sink);

    \draw[->, dashed, very thick, draw=red!75!black,
          shorten <=2pt, shorten >=2pt] (F0)--(F1);
    \draw[->, dashed, very thick, draw=yellow!85!black,
          shorten <=2pt, shorten >=2pt] (M3)--(M4);
    \draw[->, dashed, very thick, draw=red!75!black,
          shorten <=2pt, shorten >=2pt] (Ftwo7)--(Ftwo8);
    \draw[->, dashed, very thick, draw=yellow!85!black,
          shorten <=2pt, shorten >=2pt] (Mtwo10)--(Mtwo11);

    \node at ($(F0)!0.5!(F1)+(0,0.85)$)
      {\includegraphics[width=1.5cm]{loading.png}};
    \node at ($(M3)!0.5!(M4)+(0,-0.85)$)
      {\includegraphics[width=1.5cm]{unloading.png}};
    \node at ($(Ftwo7)!0.5!(Ftwo8)+(0,0.85)$)
      {\includegraphics[width=1.5cm]{loading.png}};
    \node at ($(Mtwo10)!0.5!(Mtwo11)+(0,-0.85)$)
      {\includegraphics[width=1.5cm]{unloading.png}};
\end{tikzpicture}
\end{minipage}

\vspace{0.8em}

\begin{tikzpicture}[scale=1]
\draw[loaded] (0,0) -- +(1.6,0)
    node[right=0.05cm] {Loaded trip};
\draw[->, very thick, draw=red!75!black]
    (0,-1) -- +(1.6,0)
    node[right=0.05cm] {Loading operation};
\draw[empty] (6,0) -- +(1.6,0)
    node[right=0.05cm] {Empty trip};
\draw[->, very thick, draw=yellow!85!black]
    (6,-1) -- +(1.6,0)
    node[right=0.05cm] {Unloading operation};
\draw[startend] (2.5,-2) -- +(1.6,0)
    node[right=0.05cm] {Start / End};
\end{tikzpicture}

\caption{Possible truck route combinations in $\mathcal{LTRSP}$: 
(a) a truck repeatedly travels between a single forest site and 
a mill; (b) a truck successively collects products from two 
distinct forest sites and delivers to a single mill; (c) a truck 
delivers products from two distinct forest sites to two separate 
mills. All routes begin and end at the home base.}
\label{fig:routes}

\end{figure}

In Canada, the harvested areas are geographically extensive. For the most common products, such as hardwoods, it is standard practice to load and transport full truckloads by grouping similar product species destined for the same mill. Nevertheless, in certain situations, different assortments of products may be combined within a single truckload to improve truck utilization; this is known as partial truckload, and it happens especially with the remaining objective small quantities of some assortments of products from forest sites. Furthermore, harvesting companies rely on transportation contractors, which requires accounting for a heterogeneous fleet when designing routes. Regulatory constraints impose a maximum driving time per route and contractors may be available for only a limited number of trips per day. In addition, truck availability may be restricted to specific days, as some vehicles are committed to other clients on certain days. This study explicitly considers the maximum driving time constraint and driver swapping, but for simplicity, it does not model driver rest periods. The availability of loaders at forest sites and mills is limited, which is a critical aspect for business stakeholders seeking to minimize truck waiting times at these locations. Similarly, each location operates within defined opening and 
closing times, including home bases, forest sites, and mills. To preserve the generic nature of the proposed approach, as emphasized in response to \cite{ronnqvist2015operations}, the model incorporates a set of priority rules reflecting business-driven operational constraints. In practice, certain trips are not permitted for business reasons; for example, specific products originating from given forest blocks may be exclusively assigned to designated mills. The main business rules considered in this study are the following:
\begin{itemize}
    \item[$R_1$:] All supply from a forest block is dedicated 
    to a specific mill.
    \item[$R_2$:] Some vehicles are dedicated to forest blocks 
    within a specific region.
    \item[$R_3$:] Certain mills are accessible only to specific 
    vehicle configurations.
    \item[$R_4$:] Some products from certain forest blocks are 
    dedicated to specific mills.
\end{itemize}

In real-world settings, truckers who partner with forest companies—including the industrial partner considered in this study—typically perform planning weekly, consistent with demand and supply information that is also specified weekly. More specifically, the industrial partner operates a supply network comprising approximately 50 mills that must be replenished with a variety of lumber products sourced from nearly 2270 forest blocks, of which around 30 are harvested and utilized on a daily basis. This operation relies on a heterogeneous fleet of approximately 299 log trucks with comparable capacity limits (30~GMT and 35~GMT), operated by about 59 independent transportation contractors. Trucks are classified into two main categories depending on whether they are equipped with self-loading capabilities. Within each category, four vehicle configurations are available —Quad, Tri, Tri-Drive, and B-Train— resulting in a total of eight distinct truck types. Payload capacities vary across configurations and are further affected by seasonal restrictions, road weight limits, and product densities. Truck speeds are also heterogeneous, depending on vehicle configuration, road conditions, and seasonal effects. Moreover, transportation availability is dynamic, as trucks may be temporarily leased to competitors within the forestry sector or to other industries, such as mining. Mills place weekly orders that may involve multiple product types. In total, five product families are considered: softwood, hardwood, poplar, birch, and cedar.

Through discussions with business stakeholders, two primary concerns were identified: transportation cost minimization and reduction of greenhouse gas emissions. Emissions are directly proportional to fuel consumption and are therefore closely correlated with transportation costs. The goal of this study is to establish a centralized weekly planning process that produces weekly transportation schedules, while daily planning decisions are derived from a higher-level tactical plan. Consequently, transportation activities must be coordinated across multiple days to ensure feasibility and service continuity over the planning horizon. As a result, the cost to be minimized during a week is defined as comprising three components: the first component addresses transportation costs, the second minimizes waiting time at forest sites and mills, and the third penalizes unsatisfied demand at mills.

\section{Routing representation }

Discussions with forestry companies highlighted a critical operational challenge, in particular synchronization of trips. Forest sites and mills are equipped with a limited number of loaders; therefore, to prevent excessive waiting times, the number of trucks at any site must not exceed the available loading or unloading equipment. To address this issue, we opted for a time-space formulation of the problem based on time discretization. Specifically, we divide the operating time into fixed intervals using a specific discretization value that will be highlighted in Section \ref{queue}. This enables us to introduce constraints that limit the number of trucks at each site during each time interval, along with additional operational constraints. Incorporating this time component helps to avoid queues at mills and forest sites. Figure~\ref{fig6} presents the time--space representation 
of $\mathcal{LTRSP}$ at mill $M^*$, where the operating time is divided into 12 equal intervals.
Before introducing the \textit{MILP}, we design a routing network consisting of nodes and arcs. Each node in the network corresponds to a specific location, namely a home base, a forest site, or a mill, at a given time interval $i$, thereby combining spatial and temporal dimensions within a unified representation. Arcs connect node pairs and define the incoming and outgoing flow at each location. 
At forest sites, outgoing arcs allow a truck to wait for a loader or begin loading \cite{amrouss2017real}. 
At mills, waiting and unloading are modeled separately, as illustrated in Figure~\ref{fig6} for mill~$M^*$ at time interval $i=2$, where $a_w$ and $a_u$ denote the waiting and unloading arcs respectively.

\begin{figure}[ht]
    \centering
    \begin{tikzpicture}[xscale=0.5, yscale=0.4, every node/.style={scale=0.5}]

        \draw[thick] (-0.5,-2) -- (14,-2) node[anchor=north] {Time};
        \foreach \x in {0,1,2,3,4,5,6,7,8,9,10,11,12} {
            \draw (\x,-1.8) -- (\x,-2);
            \node at (\x,-1.5) {\tiny$i=\x$};
            \node[draw,circle,fill=brown,inner sep=1.5pt] (MT\x) at (\x,-2.5) {};
        }

        \draw[thick] (-0.5,-7) -- (14,-7) node[anchor=north] {Time};
        \foreach \x in {0,1,2,3,4,5,6,7,8,9,10,11,12} {
            \draw (\x,-6.8) -- (\x,-7);
            \node at (\x,-6.5) {\tiny$i=\x$};
            \node[draw,circle,fill=brown,inner sep=1.5pt] (MB\x) at (\x,-6) {};
        }

        \node[anchor=west] at (-3,-4.25) {Mill $M^{*}$};
        \node at (-1.25,-4.1)
            {\includegraphics[width=0.9cm]{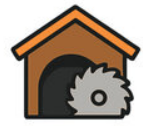}};


        \draw[->, very thick, draw=purple!80!black] (2,-6) -- (3,-6);
        \node at (2.5,-5.7) {$a_w$};

        \draw[->, very thick, draw=yellow!85!black] (2,-6) -- (3,-2.5);
        \node at (2.2,-4) {$a_u$};


        \node at (3,-1.1)
            {\includegraphics[width=0.9cm]{loading.png}};

        \node at (3,-7.8)
            {\includegraphics[width=0.9cm]{unloading.png}};

        \node[draw, fill=green, minimum size=4pt] at (3,-2.5) {};
        \node[draw, fill=green, minimum size=4pt] at (3,-6) {};

    \end{tikzpicture}

    \caption{Time--space network at mill $M^{*}$, illustrating the conventional arcs associated with loading, unloading, and waiting operations.}
    \label{fig6}
\end{figure}
Two categories of arcs are defined within the routing network. The first category, referred to as real arcs, links locations across time intervals, with associated costs reflecting travel time, which is generally proportional to distance. The second category, referred to as conventional arcs, models 
on-site operations such as loading, unloading, or waiting, with costs determined by the time required for each operation and any waiting time incurred. We only consider feasible incoming and outgoing arcs for each node in terms of travel-time consistency.

Technically, for each arc, the expected truck arrival time is computed, and any arc whose destination node falls at a time interval prior to that arrival is discarded as infeasible. Figure \ref{fig7} highlights these excluded arcs in red, since a truck travels from $M^\star$ at $i=4$ arrives at $F^\star$ at $i=8$.
\begin{figure}[H]
\centering
\resizebox{1\columnwidth}{!}{%
\begin{tikzpicture}[scale=0.8,/.style={scale=0.7}]

    \node[draw,circle,fill=brown,inner sep=2pt] (Source) at (-1,-3) {};
    \node[anchor=east] at (Source) {Home Base};

    \node[draw,circle,fill=brown,inner sep=2pt] (Sink) at (13,-3) {};
    \node[anchor=west] at (Sink) {Home Base};

    \foreach \y in {0,...,13} {
        \node[draw,circle,fill=brown,inner sep=1.5pt] (F\y) at (\y,-4) {};
    }

    \foreach \y in {0,...,13} {
        \node[draw,circle,fill=brown,inner sep=1.5pt] (M\y) at (\y,-6) {};
    }

    \node[anchor=east] at (F0) {Forest Site $F^{*}$};
    \node[anchor=east] at (M0) {Mill $M^{*}$};


    \draw[->, thick, dashed, draw=blue!70!black] (Source) -- (F0);

    \draw[->, very thick, draw=violet!75!black] (F0) -- (M3);

    \draw[->, very thick, draw=orange!85!black] (M3) -- (F7);

    \draw[->, very thick, draw=violet!75!black] (F7) -- (M10);

    \draw[->, thick, dashed, draw=blue!70!black] (M10) -- (Sink);

    \node[draw, fill=green, minimum width=3pt, minimum height=3pt] at (3,-6) {};
    \node[anchor=north] at (3,-6.2) {\small $i=4$};

    \node[draw, fill=green, minimum width=3pt, minimum height=3pt] at (7,-4) {};
    \node[anchor=north] at (7,-4.2) {\small $i=8$};

    \draw[red, thick,dashed, -{Latex[length=2mm]}] (M3) to (F1);
    \draw[red, thick, dashed, -{Latex[length=2mm]}] (M3) to (F2);
    \draw[red, thick, dashed, -{Latex[length=2mm]}] (M3) to (F3);
    \draw[red, thick, dashed, -{Latex[length=2mm]}] (M3) to (F4);
    \draw[red, thick,dashed, -{Latex[length=2mm]}] (M3) to (F5);
    \draw[red, thick, dashed, -{Latex[length=2mm]}] (M3) to (F6);

\end{tikzpicture}
}
\caption{Infeasible arcs between the mill and the forest site are shown in red and removed from the network.}
\label{fig7}
\end{figure}
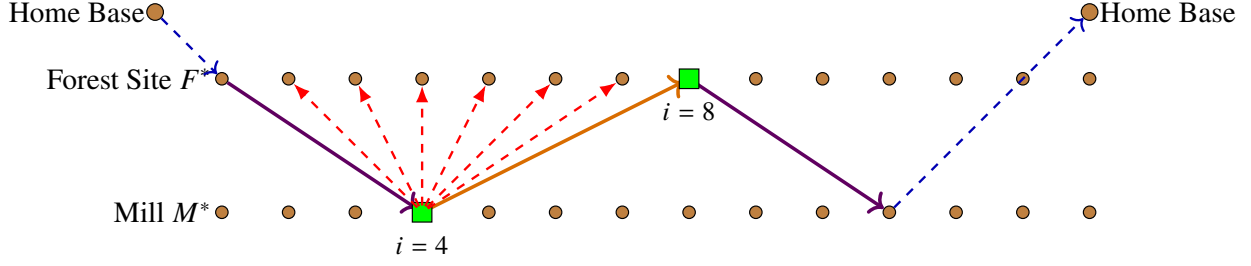

The time--space representation is modeled as a graph $G = (\mathcal{V}, A)$, where $\mathcal{V} = \{v_{ni}=(n,i) \mid n \in N, \ i \in I_n\}$ is the set of nodes, and $I_n \subseteq I$ is the set of feasible time steps for location $n$ (home bases, forests, mills). The arc set $A$ is partitioned into two subsets, namely real arcs and conventional arcs, such that $A = A_{\mathrm{real}} \cup A_{\mathrm{conv}}$. Real arcs are given by
\[
A_{\mathrm{real}} = \{(v_{n_1i_1},v_{n_2i_2}) \mid n_1, n_2 \in N, \ i_2 \geq i_1 + T_{n_1n_2}\}
\]
and represent truck movements, ensuring arrival at $n_2$ after travel time $T_{n_1n_2}$. Conventional arcs are given by
\[
A_{\mathrm{conv}} = \{(v_{ni},v_{ni+1}) \mid n \in N\}
\]
and account for waiting, loading, or unloading with service time $S_{ni}$. It is worth noting that the generation and processing of this graph explicitly
incorporate business rules such as $R_1$, $R_2$, $R_3$ and $R_4$, which prevent the
creation of arcs lacking operational relevance. As a result, the final graph
$G = (\mathcal{V}, A)$ contains only business-meaningful connections. Moreover,
this graph is clearly directed and acyclic, as formalized
in the following remark :

\begin{remark}
    \label{acyclicity-time-expanded}
    Let $G = (\mathcal{V}, A)$ be a time-expanded transportation graph where:
    \begin{itemize}
        \item Each node $v_{ni} \in \mathcal{V}$ represents a location $n $ at time $i$,
        \item Each arc $a = (v_{ni}, v_{mj}) \in A$ satisfies $i < j$.
    \end{itemize}
    Then $G$ is a directed acyclic graph.
\end{remark}

\begin{proof}
    Suppose, for the sake of contradiction, that $G$ contains a directed cycle $C = \{v_1, v_2, \dots, v_k, v_1\}$.\\
    By construction of the graph, every arc $(v_i, v_j)$ in $A$ satisfies $i < j$, which implies that time strictly increases along the path:
    \[
        t(v_1) < t(v_2) < \cdots < t(v_k) < t(v_1).
    \]
    But this leads to a contradiction: we cannot have $t(v_k) < t(v_1)$ and also $t(v_1) < t(v_k)$.\\
    Hence, such a cycle cannot exist. Therefore, $G$ contains no directed cycles, and is thus acyclic.
\end{proof}

\section{Queuing aspect within $\mathcal{LTRSP}$}
\label{queue}
As stated above, the routing representation of $\mathcal{LTRSP}$ is based on time discretization to ensure trips synchronization under limited operating equipment, e.g., loaders. This time discretization aims to minimize congestion, truck waiting times and maximize the number of trucks served. In fact, the targeted $\mathcal{LTRSP}$ produces a plan of routes, and within each route, it includes an assignment of trucks to nodes at specific time intervals $i$. For example, the following timeline illustrates the truck 
being loaded at forest site $F$ at $i=5$:

\begin{center}
\begin{tikzpicture}[scale=1, every node/.style={scale=0.6}]
    \draw[->,thick] (-0.5,-1.2) -- (14,-1.2) node[anchor=north] {Time};
    
    \foreach \x in {0,1,2,3,4,5,6,7,8,9,10,11,12} {
        \draw (\x cm,-1) -- (\x cm,-1.2);
        \node at (\x cm, -0.6) {\small $i=\x$};
    }
    
    \fill[brown!50] (5 cm,-1.2) rectangle (6 cm,-0.9);
    
    \node[scale=3.5] at (5.5, -1.8) {\textcolor{brown}{\faTruck}};
    \node at (5.5, -2.4) {\small Truck at $i=5$};
\end{tikzpicture}
\end{center}

 We study the impact of the scheduling interval between consecutive trucks. We aim to identify a fixed interval $\Delta$ (minutes) to discretize time that balances queueing delay and idle time over a planning horizon $T>0$.

Let $L_j$ be the random variable of loading time of the $j$th truck, with mean $\mathbb{E}[L_j]=\overline{L}>0$ and variance $\mathrm{Var}(L_j)=\sigma^2<\infty$. 
Trucks are scheduled deterministically every $\Delta>0$ minutes, corresponding to a $D/G/1$ setting (deterministic interarrivals, general service times, one loader). 
Over a planning horizon $T$, the number of scheduled trucks is $N(\Delta)=\lfloor T/\Delta\rfloor$. 
Let $W_j(\Delta)$ denote the waiting time of the $j$th truck under FIFO service, and define the total expected waiting time
$W(\Delta)=\sum_{j=1}^{N(\Delta)} \mathbb{E}[W_j(\Delta)]$. The corresponding \emph{traffic intensity} is $\rho(\Delta)=\overline{L}/\Delta$.

\paragraph{Lindley recursion.}
Waiting times satisfy Lindley’s equation \cite{asmussen2003applied} : $
W_{j+1}(\Delta) \;=\; \max\{0,\, W_j(\Delta) + L_j - \Delta\}$, $W_1(\Delta)=0 $. To illustrate possible scenarios for a loader utilization, let
us suppose that $L_j\equiv\overline{L}$ the recursion becomes $W_{j+1}= \max\{0, W_j+\overline{L}-\Delta\}$.
If $\Delta<\overline{L}$, then $W_j$ increases linearly (trucks backlog).
If $\Delta>\overline{L}$, then $W_j\equiv0$ but every cycle leaves $\Delta-\overline{L}$ minutes idle and number of trucks that can be served over $T$ is decreased as $\lfloor T/\Delta\rfloor \leq  \lfloor T/\bar{x}\rfloor$.
At $\Delta=\overline{L}$, $W_j\equiv0$ and no idle time occurs. Figure \ref{fig:interval} illustrates the three cases.
\begin{figure}[H]
\centering
\begin{tikzpicture}[scale=1.2]

\draw[->] (0,0) -- (10,0) node[right] {\footnotesize Time};
\draw[->] (0,0) -- (0,3.5) node[above] {\footnotesize Loader Activity};

\draw[fill=gray!20] (0.5,3) rectangle (2,3.5);
\draw[fill=gray!20] (2.5,3) rectangle (4,3.5);
\draw[fill=gray!20] (4.5,3) rectangle (6,3.5);
\draw[dashed] (0.5,3) -- (0.5,0);
\draw[dashed] (2.5,3) -- (2.5,0);
\draw[dashed] (4.5,3) -- (4.5,0);
\node at (7.5,3.25) {\footnotesize $\Delta > \overline{L}$ (Idle Time)};

\draw[fill=blue!30] (0.5,2) rectangle (2,2.5);
\draw[fill=blue!30] (2,2) rectangle (3.5,2.5);
\draw[fill=blue!30] (3.5,2) rectangle (5,2.5);
\draw[dashed] (0.5,2) -- (0.5,0);
\draw[dashed] (2,2) -- (2,0);
\draw[dashed] (3.5,2) -- (3.5,0);
\node at (7.5,2.25) {\footnotesize $\Delta = \overline{L}$ (Optimal)};

\draw[fill=red!40] (0.5,1) rectangle (2.5,1.5);
\draw[fill=red!40] (1.5,1) rectangle (3.5,1.5);
\draw[fill=red!40] (2.5,1) rectangle (4.5,1.5);
\draw[dashed] (0.5,1) -- (0.5,0);
\draw[dashed] (1.5,1) -- (1.5,0);
\draw[dashed] (2.5,1) -- (2.5,0);
\node at (7.5,1.25) {\footnotesize $\Delta < \overline{L}$ (Backlog builds)};

\end{tikzpicture}
\caption{Loader utilization across scheduling regimes.}
\label{fig:interval}
\end{figure}
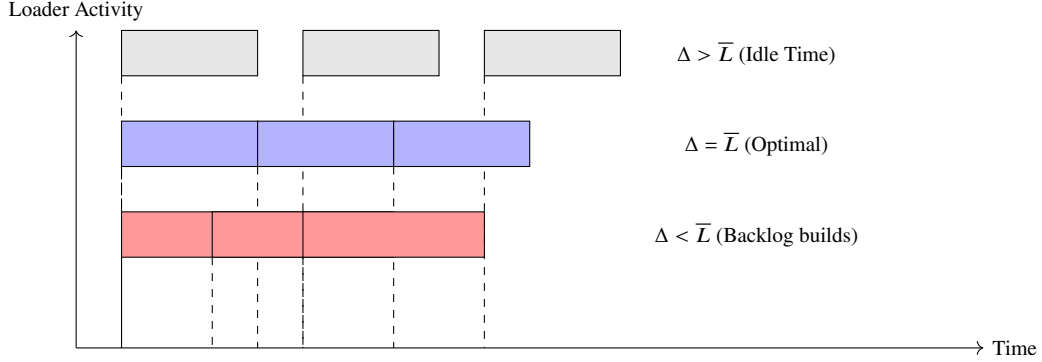

In practice, $\sigma^2 > 0$, as the forest industry is subject to considerable variability \cite{gil2016log}. 
Therefore, exact zero waiting time is impossible in steady state unless $\rho(\Delta) < 1$ (and even then, variability may still induce positive waiting). For deterministic interarrivals ($D/G/1$), a sharp and widely used approximation called Kingman’s formula \cite{kingman1961single} gives the expected waiting time per truck :
\begin{equation}
\label{eq:kingman}
\mathbb{E}[W_j(\Delta)] \;\approx\; \frac{\rho(\Delta)}{1-\rho(\Delta)} \cdot \frac{c_a^2+c_s^2}{2}\,\overline{L}
\;=\; \frac{\overline{L}}{\Delta-\overline{L}}\cdot \frac{\sigma^2}{2\,\overline{L}}
\;=\; \frac{\sigma^2}{2\,\big(\Delta-\overline{L}\big)} ,
\end{equation}
where $c_a^2=0$ (deterministic arrivals) and $c_s^2=\sigma^2/\overline{L}^2$.
In long run, the idle fraction is $1-\rho(\Delta)=1-\overline{L}/\Delta$, independent of $\sigma^2$.

Equation \eqref{eq:kingman} yields a clear trade-off: increasing $\Delta$ reduces waiting but increases idle time of loaders; decreasing $\Delta$ increases loaded trucks number but eventually increases waiting time per truck.

\vspace{0.5em}
\begin{proposition} \label{prop:frontier}To evaluate the loading process delay frontier with variability, let's fix weights $\alpha,\beta>0$ two importance weights, and consider the queuing system loss function
\[
\mathcal{C}(\Delta) \;\equiv\; 
\underbrace{\alpha\,\frac{T}{\Delta}\,\mathbb{E}[W_j(\Delta)]}_{\text{total expected waiting over }[0,T]}
\;+\;
\underbrace{\beta\,T\Big(1-\tfrac{\overline{L}}{\Delta}\Big)}_{\text{total idle time over }[0,T]}\!,
\qquad \Delta>\overline{L}.
\]
Under the Kingman approximation \eqref{eq:kingman}, the minimizer satisfies
\[
\Delta^\star \;=\; \overline{L} \;+\; s^\star,\qquad
s^\star \;=\; \frac{\frac{\alpha\sigma^2}{2} \;+\; \sqrt{\big(\tfrac{\alpha\sigma^2}{2}\big)^2 + \alpha\beta\,\overline{L}^{\,2}\sigma^2}}{\beta\,\overline{L}} \;>\; 0.
\]
\end{proposition}
In particular, as the variability vanishes ($\sigma^2 \downarrow 0$), we obtain $s^\star \downarrow 0$ and thus $\Delta^\star \to \overline{L}$. 
This recovers the intuitive result that when service times exhibit negligible variability, the time discretization can be chosen equal to the mean loading (or unloading) duration.

\begin{proof}
Using $N(\Delta)\approx T/\Delta$ and \eqref{eq:kingman}, minimize
\(
\mathcal{C}(\Delta)/T
= \alpha\,\frac{1}{\Delta}\cdot\frac{\sigma^2}{2(\Delta-\overline{L})}
+ \beta\big(1-\tfrac{\overline{L}}{\Delta}\big)
\)
over $\Delta>\overline{L}$. Setting $s=\Delta-\overline{L}>0$, algebra yields the first-order condition
\(
\beta\,\overline{L}\,s^2 = \tfrac{\alpha\sigma^2}{2}\,(\overline{L}+2s),
\)
whose positive root gives $s^\star$ as stated.
\end{proof}
Proposition~\ref{prop:frontier} characterizes the discretization interval
\(\Delta^\star = \overline{L} + s^\star\) as the unique minimizer of a
queueing-based loss function that balances total expected waiting time and
total idle time through the weights \(\alpha\) and \(\beta\).
The resulting buffer \(s^\star\) therefore depends on the relative importance
assigned to congestion (waiting) versus resource underutilization (idle time).

In practice, however, these economic weights are often difficult to calibrate
explicitly.
Operationally, planners typically reason instead in terms of an acceptable
risk level: namely, the maximum probability \(\varepsilon\) that a truck
arrives while the loader is still busy.
This motivates a complementary formulation in which the discretization interval
is chosen to satisfy a probabilistic feasibility constraint rather than to
minimize an explicit cost function. The optimal buffer in Appendix A admits the expansion
\[
s^\star
=
\sigma \sqrt{\frac{\alpha}{\beta}}
+
\mathcal{O}(\sigma^2),
\]
which reveals that the leading-order behavior of the optimal discretization
interval is governed by the ratio \(\alpha/\beta\).
Consequently, specifying a tolerance level \(\varepsilon\) implicitly fixes the
magnitude of the safety buffer required to control the tail probability
\(\mathbb{P}(L_i > \overline{L} + s)\).
Independently of any distributional assumption, such a probabilistic feasibility
constraint necessarily imposes that the buffer scales as
\(s = \sigma\, g(\varepsilon)\), where \(g(\varepsilon)\) is a dimensionless
function that increases as \(\varepsilon\) decreases.

Since Proposition~\ref{prop:frontier} shows that the optimal buffer satisfies
\(s^\star = \sigma \sqrt{\alpha/\beta} + \mathcal{O}(\sigma^2)\),
the ratio \(\alpha/\beta\) plays exactly the role of this dimensionless safety
factor.
Thus, specifying a tolerance level \(\varepsilon\) can be interpreted as
implicitly selecting a corresponding value of \(\alpha/\beta\), which we write
in the generic form
\[
\frac{\alpha}{\beta} \;\approx\; g(\varepsilon)^2,
\]
anticipating that an explicit expression for \(g(\varepsilon)\) will be derived
in the corollary below.

\begin{corollary}
\label{cor:cantelli}
For any $\varepsilon \in ]0,1[ $ choosing
$\Delta \;=\; \overline{L} \;+\; \sigma \sqrt{\tfrac{1-\varepsilon}{\varepsilon}}$ guarantees $\mathbb{P}(L_i > \Delta) \leq \varepsilon$. Hence, with probability at least $1-\varepsilon$, each truck arrives when the loader is free.
\end{corollary}

\begin{proof}
We want $\mathbb{P}(L_j > \Delta) \leq \varepsilon$.  
By Cantelli’s inequality \cite{cantelli1929sui},
\[
\mathbb{P}(L_j - \overline{L} \geq t) \;\leq\; \frac{\sigma^2}{\sigma^2+t^2}, \quad t>0.
\]
Setting $t = \Delta - \overline{L}$ gives
\[
\mathbb{P}(L_i > \Delta) \;\leq\; \frac{\sigma^2}{\sigma^2 + (\Delta-\overline{L})^2}.
\]
To enforce this bound $\leq \varepsilon$, it suffices that
\[
\frac{\sigma^2}{\sigma^2 + (\Delta-\overline{L})^2} \leq \varepsilon,
\]
which rearranges to
\[
\Delta \;\geq\; \overline{L} + \sigma \sqrt{\tfrac{1-\varepsilon}{\varepsilon}}.
\]
Thus the stated choice of $\Delta$ ensures $\mathbb{P}(L_j > \Delta) \leq \varepsilon$.
\end{proof}

\begin{remark}

In our practical case, we evaluate the average loading time as $\overline{L}=30$ minutes with standard deviation $\sigma=5$ minutes.  
If we set the tolerance to $\varepsilon=0.1$ ( at most $10\%$ of the trucks can wait), then $\Delta = 30 + 5 \sqrt{\tfrac{1-0.1}{0.1}}
= 30 + 5\times 3
\approx 45 \text{ minutes}$. Thus, spacing trucks every $\sim 45$ minutes ensures that at least $90\%$ of them find the loader available upon arrival. This value is used to determine the time interval discretization when building the routing network.
\end{remark}

\section{Problem formulation}
\label{prob-formulation}
To address the aspects discussed earlier, we formulated a MILP model to generate routes comprising sequences of loaded and unloaded trips, assign trucks to these routes, and schedule truck arrivals at each node. The model can generate up to an entire week of $T$ days, ensuring that the demands of the mills are met while adhering to multiple constraints, including resource limitations, flow conservation, time windows and other usual constraints known in forestry. An overview of the mathematical notation used throughout this paper is provided.
The model involves the following decision variables:
\begin{flalign*}
&x_{avt} = \text{1 if arc } a \text{ is traversed by vehicle } 
            v \text{ on day } t\text{, and 0 otherwise.}&\\
&\delta_{mp} = \text{unmet demand for product } p 
               \text{ at mill } m\text{.}&
\end{flalign*}
The objective function and constraints are subsequently described. 
The mathematical model capturing the most frequent planning aspects 
encountered by our industrial partner, denoted $\mathcal{P}_m$, 
comprises Equations~(\ref{objfun})--(\ref{const14}).

{\small
\begin{align}
\min \quad & \sum_{t \in T}\sum_{v \in V}\sum_{a \in A} (c_{a}^{h}+c_{a}^{w}) 
             x_{avt} + \sum_{m \in M} \sum_{p \in P} c_{mp} \delta_{mp} 
             \label{objfun} \\[5pt]
\text{s.t.} \quad 
& \sum_{a \in A^{+}(h)} x_{avt} \leq 1, 
  \quad \forall\, v \in V,\, h \in \mathcal{HB}_v,\, t \in T 
  \label{const1} \\[5pt]
& \sum_{a \in A^{-}(h)} x_{avt} = \sum_{a \in A^{+}(h)} x_{avt}, 
  \quad \forall\, v \in V,\, h \in \mathcal{HB}_v,\, t \in T 
  \label{const2} \\[5pt]
& \sum_{a \in A^+(n)} x_{avt} = \sum_{a \in A^-(n)} x_{avt}, 
  \quad \forall\, v \in V,\, n \in N,\, t \in T 
  \label{const3} \\[5pt]
& \sum_{t \in T}\sum_{f \in F}\sum_{v \in V}\sum_{a \in A_{f}^{mp}} 
  q_v x_{avt} + \delta_{mp} = d_{mp}, 
  \quad \forall\, m \in M,\, p \in P 
  \label{const4} \\[5pt]
& \sum_{t \in T}\sum_{m \in M}\sum_{v \in V}\sum_{a \in A_{f}^{mp}} 
  q_v x_{avt} \leq s_{fp}, 
  \quad \forall\, f \in F,\, p \in P 
  \label{const5} \\[5pt]
& \sum_{t \in T}\sum_{m \in M}\sum_{f \in F}\sum_{p \in P}
  \sum_{a \in A_{f}^{mp}} x_{avt} \leq K_v, 
  \quad \forall\, v \in V 
  \label{const6} \\[5pt]
& \sum_{t \in T}\sum_{v \in V}\sum_{a \in A_{fi}^L} 
  \beta_v x_{avt} \leq \Lambda_{fi}, 
  \quad \forall\, i \in I,\, f \in F,\, t \in T 
  \label{const7} \\[5pt]
& \sum_{t \in T}\sum_{v \in V}\sum_{a \in A_{mi}^U} 
  \beta_v x_{avt} \leq \Lambda_{mi}, 
  \quad \forall\, i \in I,\, m \in M 
  \label{const8} \\[5pt]
& \tau_a^{+} - \textit{BigM}(1 - x_{avt}) + S_t \leq T^c_{an}, 
  \quad \forall\, v \in V,\, a \in A,\, t \in T 
  \label{const9} \\[5pt]
& T^o_{an} \leq \tau_a^{+} + \textit{BigM}(1 - x_{avt}), 
  \quad \forall\, v \in V,\, a \in A,\, t \in T 
  \label{const10} \\[5pt]
& x_{avt} \in \{0, 1\}, 
  \quad \forall\, v \in V,\, a \in A,\, t \in T 
  \label{const11} \\[5pt]
& \delta_{mp} \in \{0, \ldots, d_{mp}\}, 
  \quad \forall\, m \in M,\, p \in P 
  \label{const14}
\end{align}}
 The objective function \eqref{objfun} minimizes the  total transportation and waiting costs across all trips.  Constraint \eqref{const1} limits each truck to a single  departure arc. Constraints \eqref{const2} and \eqref{const3}  impose flow conservation at every node, ensuring route  continuity throughout the network. Constraint \eqref{const4}  requires that mill demand is fully met, whereas constraint  \eqref{const5} prevents the consumption of products from  exceeding available supply at forest sites. Constraint  \eqref{const6} bounds the number of trips that each truck  may perform. Constraints \eqref{const7} and \eqref{const8}  enforce loader capacity limits at forest sites and mills,  respectively. Constraints~\eqref{const9} and~\eqref{const10}  guarantee that operations at each node remain within the  prescribed time windows. Finally, constraints~\eqref{const11}  and~\eqref{const14} specify the domain of the decision  variables. This model $\mathcal{P}_m$, is the principal formulation used in collaboration with our industrial partner, as it encompasses all its common operational constraints. However, in forestry operations—both in Canada and worldwide—several additional constraints are frequently encountered. Since our objective is to introduce a comprehensive formulation. In what follows, we present the most important and critical constraints identified in \cite{audy2023planning}, and we highlight how each of them can be integrated into our formulation through explicit mathematical constraints.

\subsection{Multiple home bases \& incomplete route }

This work addresses the aspect of multiple home bases for a vehicle. 
In practice, a vehicle can be associated with more than one home base, 
for example, a contractor depot and the driver’s personal residence, 
both of which may serve as valid end-of-day locations. 
Such flexibility must be incorporated into the planning of truck routes. 
To model this aspect, we introduce a new decision variable $y_{vht}$, 
which equals 1 if truck $v$ finishes its operations on day $t$ at home base $h$, 
and 0 otherwise. 
Consequently, constraints \eqref{const1}--\eqref{const2} in $\mathcal{P}_m$ 
are replaced by the following set of constraints:

\begin{align}
&\sum_{a \in A^{+}(h_v^0)} x_{av0} \le 1 \, \forall v \in V, \label{const15} \\[5pt]
&\sum_{h \in \mathcal{HB}_v}\sum_{a \in A^{-}(h)} x_{avt} =\sum_{h \in \mathcal{HB}_v}\sum_{a \in A^{+}(h)} x_{avt} \, \forall v \in V,\ \forall t \in T, \label{const16} \\[5pt]
&y_{vht} =\sum_{a \in A^-(h)} x_{avt} \,\forall v \in V,\ \forall h \in \mathcal{HB}_v,\ \forall t \in T, \label{const17} \\[5pt]
&\sum_{h \in \mathcal{HB}_v} y_{vht}=1 \, \forall v \in V,\ \forall t \in T, \label{constt18} \\[5pt]
&y_{vht} \le y_{vh,t+1} + \sum_{a \in A^+(h)} x_{av,t+1} \,\forall v \in V,\ \forall h \in \mathcal{HB}_v,\ \forall t \in T \setminus \{t_f\}. \label{constt19}
\end{align}

Constraint~\eqref{const15} ensures that each truck departs at most once from its initial home base 
at the beginning of the planning horizon. 
Constraint~\eqref{const16} guarantees that, for each day, the number of arrivals and departures 
at the set of home bases assigned to a vehicle remain balanced. 
Constraints~\eqref{const17}--\eqref{constt19} explicitly track the daily location of each vehicle 
at a home base: \eqref{const17} defines whether a truck ends its day at a specific home base, 
\eqref{constt18} ensures that every truck is parked in exactly one home base at the end of each day, 
and \eqref{constt19} enforces continuity, requiring that a truck either remains in the same home base 
when inactive or departs from it if it resumes operations the next day.

Incomplete routes correspond to the case where trucks can remain loaded overnight 
and spend the night either at the company’s home base or at the driver’s personal residence. 
This aspect is first handled during the generation of the routing network by creating 
$A_{\mathrm{1}} = \{ ((n_1,i_1),(n_2,i_2)) \mid n_1 \in F, \ n_2 \in HB \}$, 
which represents arcs from forest blocks to return locations, 
and 
$A_{\mathrm{2}} = \{ ((n_1,i_1),(n_2,i_2)) \mid n_1 \in HB, \ n_2 \in M \}$, 
which represents arcs from these return locations to mills to complete the route on the next day. 
The subset \( A_{\mathrm{1}} \cup A_{\mathrm{2}} \) can be preprocessed using the operator 
\textit{Preprocess}\(\,(A_{\mathrm{1}} \cup A_{\mathrm{2}})\) in order to generate only meaningful arcs. 

To take into account this aspect, we add constraints \eqref{constHB1L} and \eqref{constHB3b}. 
Explicitly, we add a new variable $y^{L}_{vht}$ equal to 1 if vehicle $v$ ends its day $t$ at home base $h$ loaded, 0 otherwise:
\begin{align}
&y^{L}_{vht} 
= \sum_{a \in A_1(h)} x_{avt} \, \forall v \in V,\ \forall h \in \mathcal{HB}_v,\ \forall t \in T, \label{constHB1L} \\[5pt]
&y^{L}_{vht} \le \sum_{a \in A_2(h)} x_{av,t+1} \,\forall v \in V,\ \forall h \in \mathcal{HB}_v,\ \forall t \in T \setminus \{t_f\}. \label{constHB3b}
\end{align}

Constraint~\eqref{constHB1L} identifies whether a vehicle finishes its day at a home base while still loaded, 
based on arcs in $A_1$. 
Constraint~\eqref{constHB3b} enforces continuity of incomplete routes: 
if a vehicle ends a day loaded at a home base, it must start the following day with an arc in $A_2$, 
i.e., by traveling to a mill to unload.

\subsection{Partial Load (LTL)}

In timber transportation, it is common sometimes to find that the amount of a product available at a supply site is smaller than the truck capacity. This occurs, for example, when a harvesting operation is nearing completion and only a small residual pile remains, or when multiple sites each have partial quantities of the same product that must be cleared. In such cases, \emph{Less Than Truckload} (LTL) consolidation is necessary: several partial loads can be combined in a single trip, provided the truck capacity is not exceeded. The industrial partner often requires that these residual piles be removed promptly to free storage areas or to complete harvesting in a given block. LTL consolidation in this context is challenging because it not only impacts route optimization but also interacts with truck capacity planning and harvesting site inventory management. Poor planning of LTL pickups may lead to inefficient trips, with many partial pickups increasing travel time and operational costs. In practice, however, the number of LTL items remains limited: at most five residual LTLs are typically observed per week, and within a given forest block there is usually no more than one LTL item per product. 

We extend the principal model $\mathcal{P}_m$ to allow the inclusion of residual LTL items. Let $\mathcal{R} \subseteq F \times P$ denote the set of residual LTL items, where each element $(f,p) \in \mathcal{R}$ represents a partial quantity $r_{fp} > 0$ (in GMT) of product $p$ available at forest $f$. The principal binary routing variable $x_{avt}$ is augmented with two new variables: $x^{F}_{avt}$, equal to $1$ if arc $a$ of truck $v$ on day $t$ carries a full truckload of size $q_v$, and $x^{L}_{avtfp}$, equal to $1$ if arc $a$ of truck $v$ on day $t$ carries the LTL item $(f,p)$ of volume $r_{fp}$.

The following additional constraints are introduced. First, the load type is partitioned so that each loaded arc is either a full truckload or the corresponding LTL item:
\begin{align}
x^{F}_{avt} \;+\; x^{L}_{avtfp}
&= x_{avt} && \forall v\in V,\ \forall t\in T,\ \forall (f,p)\in\mathcal{R},\ \forall a\in A^{F}_{fp}, \label{const18} \\
r_{fp}\,x^{L}_{avtfp}
&\le q_v\,x_{avt} && \forall v\in V,\ \forall t\in T,\ \forall (f,p)\in\mathcal{R},\ \forall a\in A^{m}_{fp}. \label{const19}
\end{align}

Second, the demand satisfaction and forest supply constraints of the principal model are updated so that the quantity contributions are expressed in GMT, with the FTL part contributing $q_v$ per full load arc and the LTL part contributing the exact residual volumes:
\begin{equation}
\sum_{f\in F}\sum_{v\in V}\sum_{t\in T}\sum_{a\in A^{mp}_{f}} q_v\,x^{F}_{avt}
\;+\;
\sum_{v\in V}\sum_{t\in T}\sum_{a\in A^{mp}_{f}} r_{fp}\,x^{L}_{avtfp}
\;+\;\delta_{mp}
= d_{mp},
\qquad \forall m\in M,\ \forall p\in P,
\label{const20}
\end{equation}
\begin{equation}
\sum_{m\in M}\sum_{v\in V}\sum_{t\in T}\sum_{a\in A^{mp}_{f}} q_v\,x^{m}_{avt}
\;+\;
\sum_{v\in V}\sum_{t\in T}\sum_{a\in A^{m}_{fp}} r_{fp}\,x^{L}_{avtfp}
\;\le\; s_{fp},
\qquad \forall f\in F,\ \forall p\in P.
\label{const21}
\end{equation}

Finally, each residual LTL item must be transported exactly once, with its destination determined by the selected loaded arc:
\begin{equation}
\sum_{v\in V}\sum_{t\in T}\sum_{a\in A^{F}_{fp}} x^{L}_{avtfp} \;=\; 1,
\qquad \forall \, (f,p)\in\mathcal{R}. 
\label{const22}
\end{equation}

Constraints \ref{const18}, \ref{const19} enforce capacity feasibility for LTL arcs, while \ref{const20}, \ref{const21} ensure correct integration of LTL volumes into the flow balance. All other routing, timing, and operational constraints from the principal model remain unchanged.

Adding these constraints requires augmenting the initial routing network with new arcs allowing trucks to travel from a forest block to another block to consolidate possible partial loads:
$A_{\mathrm{3}} = \{ ((n_1,i_1),(n_2,i_2)) \mid n_1 \in \mathcal{F}_{\mathcal{R}},\ n_2 \in \mathcal{F}_{\mathcal{R}} \}$
where
$\mathcal{F}_{\mathcal{R}} = \{ f \in F \mid \exists \, p \in P \ \text{s.t.} \ (f,p) \in \mathcal{R} \}$
is the set of forest sites containing residual LTL items.

\subsection{Priority Rules}

In forest transportation operations, \emph{priority rules} represent contractual, operational, or inventory-based requirements that must be respected when assigning loads, routing trucks, or scheduling trips. A first category is \emph{destination-specific supply rules}, where certain forest products from a given forest site must be delivered only to a specific subset of mills due to contractual agreements or inventory policies. For a product $p$ from forest site $f$ allowed only to mills $\mathcal{M}_{fp} \subseteq M$, this can be enforced by
\begin{equation}
    x_{avt} = 0, 
    \quad \forall \, v \in V, \,\, a, \,\, t\, \in T \in A^F_{fp} \ \text{s.t.} \ \mathrm{end}(a) \notin \mathcal{M}_{fp}
\end{equation}
A second category is the \emph{minimum delivery requirement}, where certain mills must receive a minimum quantity of specific products during a given day. If $L_{mpt'}$ denotes the minimum required quantity (in GMT) of product $p$ at mill $m$ upon day $t'$, this can be formulated as
\begin{equation}
    \sum_{t=0}^{t'}\sum_{f \in F} \sum_{v \in V} \sum_{\substack{a \in A^{mp}_{fp} \\ }} q_v \, x^F_{avt}
    + \sum_{t=0}^{t'}\sum_{v \in V} \sum_{\substack{a \in A^{mp}_{f} \\ }} r_{fp} \, x^L_{avtfp} 
    \ge L_{mpt'}
\end{equation}
These constraints either filter infeasible arcs and enforce delivery quotas, without altering the main routing and timing structure of the model. 
\subsection{Driver Change Requirement}
In long-haul transportation, regulations on driving hours require a driver to rest after reaching a maximum allowable continuous driving time, denoted by $H^{\max}$. To model this, our framework allows for a \emph{driver change}, where a truck stops at a designated location and a relief driver takes over, effectively resetting the driving clock. Let $N^{DC}_v$ denote the set of nodes where a driver change is permitted for truck $v$.

To enforce this time limit, we introduce a continuous variable $u_{v,(n,i)} \ge 0$ to track the cumulative driving time of the driver for truck $v$ upon arrival at node-time pair $(n,i)$. A binary variable $y^{DC}_{vi} \in \{0,1\}$ is used to indicate whether a driver change occurs for truck $v$ during time slot $i$. For any arc $a$ in the network traversed by truck $v$, let $\mathrm{Start}(a)=(n_1,i_1)$ and $\mathrm{End}(a)=(n_2,i_2)$, with an associated pure driving time of $T^{a}_{v}$. The following constraints govern the propagation and reset of the cumulative driving time.

\begin{align}
& u_{v,(n_2,i_2)} \ge u_{v,(n_1,i_1)} + T^{a}_{v} - H^{\max}\,y^{DC}_{vi_2} - \text{BigM}\,(1-x_{avt}) 
&& \forall v \in V, a \in A, t \in T \label{const:dc_lower} \\[5pt]
& u_{v,(n_2,i_2)} \le u_{v,(n_1,i_1)} + T^{a}_{v} + \text{BigM}\,y^{DC}_{vi_2} + \text{BigM}\,(1-x_{avt}) 
&& \forall v \in V, a \in A, t \in T \label{const:dc_upper} \\[5pt]
& 0 \le u_{v,(n,i)} \le H^{\max}(1-y^{DC}_{vi}) + \varepsilon\,y^{DC}_{vi} 
&& \forall v \in V,\ (n,i) \in N \times T \label{const:dc_bounds} \\[5pt]
& y^{DC}_{vi} \le \sum_{n' \in N^{DC}_v} \sum_{a \in A_{n'i}} x_{avt} 
&& \forall v \in V,\ i \in T \label{const:dc_location} \\[5pt]
& u_{v,\mathrm{Start}(v)} = 0 && \forall v \in V \label{const:dc_initial}
\end{align}

Constraints \eqref{const:dc_lower} and \eqref{const:dc_upper} manage the update of the cumulative driving time along a traversed arc $a$. These constraints are activated by the term $\text{BigM}\,(1-x_{avt})$ only if truck $v$ actually uses the arc (i.e., if $x_{avt}=1$). Constraint \eqref{const:dc_bounds} enforces the driving time limits. If no driver change occurs ($y^{DC}_{vi}=0$), it ensures the cumulative time does not exceed the maximum limit $H^{\max}$. Conversely, if a driver change is performed ($y^{DC}_{vi}=1$), this constraint forces $u_{v,(n,i)}$ to be reset to a very small value $\varepsilon$. Constraint \eqref{const:dc_location} provides the logical link ensuring that a driver change can only be initiated ($y^{DC}_{vi}=1$) if truck $v$ is physically present at an authorized driver change node ($n' \in N^{DC}_v$) during time slot $i$. Finally, constraint \eqref{const:dc_initial} initializes the cumulative driving time to zero for all trucks at the start of their routes.

\section{Solution methodology}
In this section, we present the solution approaches used to solve the main mathematical model of the $\mathcal{LTRSP}$. 
To address the problem efficiently, we combine state-of-the-art MIP solvers with the well-established 
\emph{Relax \& Fix – Fix \& Optimize (R\&F–F\&O)} metaheuristic framework, 
which has recently gained increasing attention and enables the computation of high-quality solutions within reduced computational time \cite{wang2023relax}.
Before introducing the complete solution methodology, we summarize below the key empirical and theoretical observations that motivate and guide the choice of solving strategies:

\begin{itemize}
    \item \emph{Small integrality gap:} Preliminary solution tests of $\mathcal{P}_m$ reveal that LP relaxation is consistently strong as highlighted in Appendix B: only a limited subset of routing variables remains fractional at optimality.
    This confirms that the formulation exhibits a tight polyhedral structure and provides branch-and-bound with high-quality lower bounds.

    \item \emph{Early-incumbent latency:}  
 Given the size of the routing network, the main computational challenge is not the quality of the bound, but the time required to construct the first feasible schedule that satisfies all operational constraints in the planning horizon \(T\). Once such a schedule is obtained, the optimality gap closes in a reasonnable runtime.
 \end{itemize}
 
Given the two preceding points, one can observe that the \(\mathcal{LTRSP}\) formulation has a strong LP relaxation, while the main difficulty lies in generating the first feasible incumbent due to the problem size. \textit{R\&F - F\&O} framework is therefore well suited for this formulation: by decomposing the model into sequential integer blocks, it mitigates the early-incumbent latency and allows the solver to obtain a first feasible solution more quickly on partially relaxed subproblems.
This approach is further justified by the small integrality gap, since relaxation do not move the solution far from integrality.
As a result, solution quality remains stable and the search stays close to the true integer optimum.

In what follows, we describe in detail the main components of the solving approaches, namely the state-of-the-art MILP solver and the \textit{R\&F–F\&O} framework.

\subsection{MILP Solver Considerations}
The performance of branch-and-bound for large-scale vehicle routing and scheduling problems critically depends on the strength of the LP relaxation and the solver’s ability to navigate the search tree efficiently. Let $\mathrm{OPT}_{\text{LP}}$ denote the value of the LP relaxation and $\mathrm{OPT}_{\text{IP}}$ the value of the integer optimum. The integrality gap
\[
GAP_{int}= \frac{\mathrm{OPT}_{\text{IP}} - \mathrm{OPT}_{\text{LP}}}{\mathrm{OPT}_{\text{IP}}}
\]
governs the branching effort. A small $GAP_{int}$ implies that only a limited number of nodes need to be explored before proving optimality. Once an incumbent solution of quality $\mathrm{INC}$ is identified. Obtaining the initial incumbent is particularly challenging, as the solver relies on primal heuristics and node relaxations to construct high-quality feasible solutions at early stages. To improve this phase, we configure the LP algorithm inside the branch-and-bound tree by enforcing the barrier method through \texttt{Method = 2}. 
While the dual simplex method ($m=1$) efficiently reuses previously computed bases, the barrier method is better suited for large and highly degenerate relaxations.

Further, two parameter settings are employed to reduce branching inefficiencies. By setting \texttt{DegenMoves = 0}, we eliminate degeneracy-based pivot moves that add computational overhead without improving bounds in highly degenerate relaxations. Moreover, disabling crossover (\texttt{Crossover = 0}) avoids reconstructing simplex bases after barrier solves, thereby reducing unnecessary post-processing time within each node.

Finally, to mitigate the tailing-off effect inherent in MIP solvers, we impose a stopping criterion on the optimality gap. Let $\Delta(k)$ denote the relative gap at iteration $k$, defined as
\[
\Delta(k) = \frac{\mathrm{INC}(k) - \mathrm{LB}(k)}{\mathrm{INC}(k)},
\]
where $\mathrm{INC}(k)$ and $\mathrm{LB}(k)$ denote the incumbent and the best lower bound, respectively. It is well known that once $\Delta(k)$ falls below $5\%$, further progress is dominated by improvements in $\mathrm{LB}(k)$ rather than $\mathrm{INC}(k)$. Thus, enforcing \texttt{MIPGap = 0.01} guarantees termination once $\Delta(k) \leq 1\%$, ensuring solutions are practically optimal while avoiding excessive computational time in late-stage search.

In summary, the combination of a barrier-based LP solution strategy, suppression of degeneracy moves, disabling crossover, and an early-gap termination rule provides an effective configuration for scaling MILP solvers to horizon-based $\mathcal{LTRSP}$ instances. These adjustments exploit both the structural strength of the formulation and solver-specific mechanisms to alleviate branching effort.

\subsection{Relax-and-Fix (\textit{R\& F}) Strategy}

\textit{R\&F} is an iterative decomposition technique designed to progressively construct a feasible and near-optimal solution by solving a sequence of restricted MILP subproblems \cite{brahimi2015integrating}. It is particularly effective for time-dependent \textit{MILPs}, where decisions evolve over a finite planning horizon.
In our formulation, we divide the time horizon \( T \) into nonoverlapping blocks \( B \), where each block consists of a subset of consecutive time periods. The key idea is to iteratively fix decision variables for past periods, solve a subset of integer variables in the current time block, and relax future decisions. At each iteration \( k \), the problem is reformulated by partitioning the decision variables into three sets. The variables corresponding to previous time blocks \( \{ B_1, B_2, \dots, B_{k-1} \} \) are fixed to their integer values determined previously, ensuring consistency with previous decisions. The variables within the current time block \( B_k \) remain integer and are solved exactly, allowing the model to enforce discrete decisions over this subset. Finally, the variables associated with future time blocks \( \{ B_{k+1}, \dots, B_T \} \) are relaxed to their continuous counterparts, reducing the combinatorial complexity of the problem while preserving flexibility for subsequent iterations. Mathematically, this can be expressed as:
\[
\begin{cases}
    x_i \in \{0,1\} \quad \forall \,\,i \in B_k
     \\x_i = x_i^* \in \{0,1\} \quad \forall \,\, i \in \bigcup_{j=1}^{k-1} B_j 
     \\ x_i \in [0,1] \quad \forall i \in \bigcup_{j=k+1}^{T} B_j
\end{cases}
\]

This structured approach ensures that the integer assignments are constructed progressively while maintaining computational tractability. This ensures that the size of the problem remains manageable, significantly reducing computational effort compared to solving the whole \textit{MILP} at once. The entire approach is formally described in 
Algorithm~\ref{alg:relax_fix}.

\begin{algorithm}[H]
\caption{Relax-and-Fix}
\begin{algorithmic}[1]
\REQUIRE MILP model with planning horizon $T$, block size $B$
\ENSURE Near-optimal solution to the MILP
\STATE Load the MILP model and configure the objective function 
       to minimize total cost
\STATE Declare decision variables and incorporate all relevant 
       constraints into the model
\STATE Initialize storage structures for solutions and 
       the list of variables to be fixed
\FOR{each block $b$ from $0$ to $T$ with step $B$}
    \STATE \textbf{Relax-and-Fix strategy:}
    \begin{itemize}
        \item Fix variables associated with previously solved blocks
        \item Treat variables of the current block as integer
        \item Relax variables of upcoming blocks to continuous
    \end{itemize}
    \STATE Solve the resulting sub-problem using an MILP solver
    \IF{a feasible solution is found}
        \STATE Record the solution for the current block
    \ELSE
        \STATE \textbf{Stop:} No feasible solution exists for 
               the current block
    \ENDIF
\ENDFOR
\STATE \textbf{Output:} Final sol
ution status and objective value
\end{algorithmic}
\label{alg:relax_fix}
\end{algorithm}

\subsection{Fix-and-Optimize (F\&O) Strategy}

Once an initial solution is obtained using \textit{Relax-and-Fix}, we apply a \textit{Fix-and-Optimize} strategy to iteratively improve the integer solution or removing possible infeasibilities by focusing on small, overlapping subsets of decision variables. Unlike \textit{Relax-and-Fix}, which progressively assigns integer values to variables over time, \textit{Fix-and-Optimize} iteratively refines the solution by selecting subsets of decision variables for re-optimization while keeping others fixed as described by Algorithm \ref{Fix and optimize}. At each iteration \( k \), the optimization is restricted to a rolling time window \( I_k \), where the integer variables within the window are re-optimized, while all other variables remain fixed at their previously assigned values. This can be mathematically formulated as:
\[
x_i = x_i^*, \quad \forall\, i \in \bigcup_{j=1,\, j \neq k}^{T} I_j
\]
where \( x_i^* \) represents the previously assigned integer value of \( x_i \). This method ensures that only a small subset of variables is modified at each iteration, maintaining feasibility while improving the objective function in a controlled manner. By iteratively selecting different time windows \( I_k \), the \textit{Fix-and-Optimize} approach enhances solution quality without significantly increasing computational complexity.

This approach ensures a controlled exploration of the solution space, allowing for localized improvements without significantly altering feasibility. By iteratively refining subsets of decision variables, \textit{Fix-and-Optimize} enhances solution quality while maintaining computational efficiency.
\begin{algorithm}[H]
\caption{Fix-and-Optimize }
\label{alg:fix_optimize}
\begin{algorithmic}[1]
\REQUIRE MILP model with time horizon \( T \), set of time Intervals \( I_k \)
\ENSURE Improved solution to the MILP

\STATE Initialize the MILP model and set the objective to minimize total cost
\STATE Define decision variables and add relevant constraints to the model
\STATE Initialize solution storage and best-known objective value

\STATE Set \( k = 1 \)

\FOR{each block $I_k$}
    \STATE \textbf{Fix-and-Optimize Strategy:}
    \begin{itemize}
        \item Fix variables outside the selected interval \( I_k \)
        \item Optimize integer variables within \( I_k \)
    \end{itemize}
    \STATE Solve the restricted MILP using the solver
    \IF{solution improves the objective value}
        \STATE Update the best-known solution
    \ENDIF
    \STATE Move to the next time window: \( k \gets k+1 \)
\ENDFOR

\STATE \textbf{Output:} Final solution status, objective value, and selected variable values

\end{algorithmic}
\label{Fix and optimize}
\end{algorithm}
\subsection{Acceleration strategies}
\label{sec:acceleration_strategies}
The proposed formulation $\mathcal{P}_m$ involves a heterogeneous fleet,
time-dependent routing decisions, and detailed service-time modeling, which
collectively lead to a large and highly symmetric search space. Although the
core formulation is theoretically sound, practical solvability strongly depends
on additional modeling choices that influence relaxation strength and solver
behavior.
To improve computational efficiency without altering the feasible region or
excluding optimal solutions, we incorporate two acceleration strategies directly
into the formulation: (i) symmetry-breaking constraints for identical vehicles 
of the same contractor, and (ii) indicator constraints to replace big-$M$ 
time-window formulations.

\paragraph{Symmetry breaking.}
For a given contractor $c \in C$, let $V_c$ denote the set of vehicles operated
by $c$, and let $V_{c}^{\kappa} \subseteq V_c$ be the subset of vehicles sharing
the same configuration $\kappa$. Vehicles in $V_{c}^{\kappa}$ are
interchangeable, which induces a high degree of symmetry in the formulation.
This symmetry allows the solver to generate multiple equivalent solutions that
differ only by a permutation of identical vehicles, unnecessarily inflating the
branch-and-bound tree. To mitigate this effect, we introduce symmetry-breaking
constraints that impose a partial ordering on vehicle utilization within each
group $V_{c}^{\kappa}$.
Let $V_{c}^{\kappa} = \{v_1, v_2, \ldots, v_{k_c}\}$ be indexed arbitrarily.
We enforce the following ordering constraints:
\begin{equation}
\sum_{a \in A} \sum_{t \in T} x_{av_i t}
\;\ge\;
\sum_{a \in A} \sum_{t \in T} x_{av_{i+1} t},
\quad \forall c \in C,\ \forall \kappa,\ i = 1,\ldots,k_c-1,
\label{eq:symmetry_breaking}
\end{equation}
where $x_{avt}$ is the binary variable indicating whether arc $a$ is traversed
by vehicle $v$ at time period $t$.
Constraints \eqref{eq:symmetry_breaking} ensure that lower-indexed vehicles are
used at least as much as higher-indexed ones, thereby eliminating equivalent
solutions obtained through vehicle permutations.

\paragraph{Indicator constraints.}
Time-window feasibility at nodes is enforced only when the corresponding routing
arc is selected. Constraints \ref{const9}, \ref{const10} weaken the
linear relaxation and may cause numerical instability when $M$ is large. To
address this issue, we replace them with indicator constraints that directly
encode the underlying logical implication:
\begin{align}
x_{avt} = 1 &\;\Rightarrow\; \tau^+_{a} + S_a \le T^C_n, \label{eq:ind_close}\\
x_{avt} = 1 &\;\Rightarrow\; \tau^+_{a} + S_a \ge T^O_n, \label{eq:ind_open}
\end{align}
for all $a \in A$, $v \in V$, and $t \in T$.
Indicator constraints activate the time-window conditions only when arc $a$ is
selected by vehicle $v$ at time $t$, and remain inactive otherwise. This
reformulation preserves the exact role of the original constraints while
avoiding the use of constants. Conventional solvers such as \textit{Cplex} 
and \textit{Gurobi} support these kind of constraints.
\section{Computational study}
\label{exper}
In this section, we report the computational experiments conducted on
real-world data provided by a Canadian forest company. Recall that, in the Canadian forestry context, routing and scheduling
planning are typically made on a weekly basis.
Accordingly, the demand and supply information of our industrial partner
are specified at a weekly level.
The proposed implementation is general and parametrized by the planning
horizon $T$ enabling the generation of routing and scheduling plans for
arbitrary horizons.

To illustrate the outcomes of our study, we propose the following test plan. First, we provide details about the instances used, and then we present weekly routing and scheduling results (i.e., $T = 5$). As will be explained later, for weekly planning we compare different solving methods in terms of computational runtime and solution quality. Finally, we conduct a post-computational analysis to discuss the obtained results and derive business insights.
\subsection{Context and instances description}
We obtained two years of historical operational data from our industrial partner. The full dataset spans a network of 2 379 distinct locations, consisting of 50 mills, 2 270 forest blocks, and 59 contractor facilities. From these data, we generated weekly instances containing detailed information on harvested forest sites, mill demands, contractor fleets, loaders availability, and service times for loading and unloading. In the following subsections, we provide a detailed characterization of these instances. Table \ref{instances-characteristics} reports the set of weekly routing instances and their main characteristics.

\begin{figure}[htbp]
    \centering
    \includegraphics[width=0.9\linewidth]{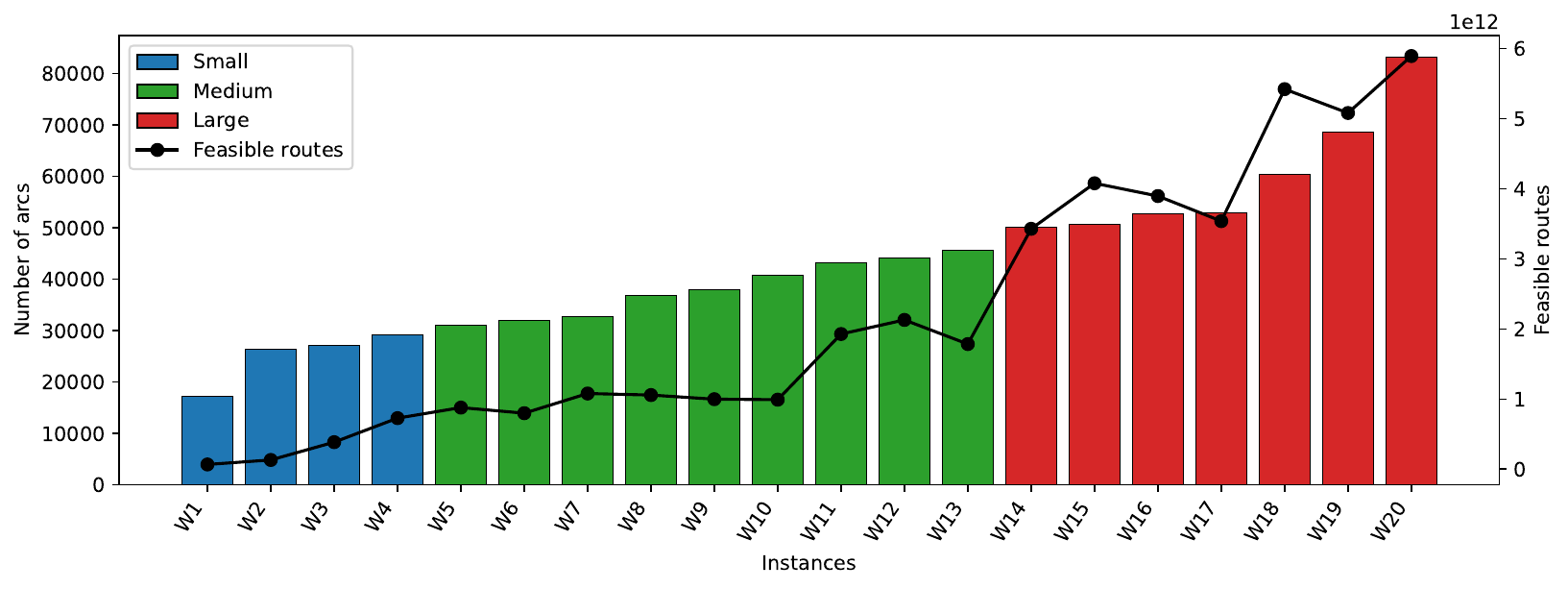}
    \caption{Growth of feasible routes as a function of the number of arcs.
    The number of arcs serves as a proxy for network density, which directly affects computational complexity.}
    \label{fig:feasible_routes}
\end{figure}

Based on the number of arcs, the instances $\mathcal{W}$ are categorized into three groups: small $\mathcal{W}_s$, medium $\mathcal{W}_m$, and large $\mathcal{W}_l$. 
The number of arcs is used as a proxy for network density, i.e., the number of potential routes connecting supply and demand nodes. 
Density is the main factor influencing computational effort, as it increases symmetry and degeneracy, thereby directly affecting the branching effort. 
In addition, our problem structure is strongly impacted by other factors such as a heterogeneous truck fleet (with different capacities and loading equipment), multiple product types, and diverse client demands. 
To illustrate, Figure~\ref{fig:feasible_routes} shows that the routing network size grows exponentially with the number of arcs: 
even a small increase in arcs leads to an exponential increase in the number of feasible routes.

\begin{table}[H]
\centering
\caption{Characteristics of the 20 weekly instances categorized by number of arcs. 
\textbf{Instance}: label of the weekly instance; 
\textbf{Dates}: planning horizon of the instance; 
\textbf{Served Mills}: number of mills served; 
\textbf{Used Blocks}: number of forest blocks used; 
\textbf{Products}: number of product types; 
\textbf{Vehicles}: number of available vehicles; 
\textbf{Home bases}: number of truck home bases; 
\textbf{Nodes}: number of nodes in the network; 
\textbf{Arcs}: number of arcs in the network; 
\textbf{Total Demand (GMT)}: total demand expressed in GMT units; 
\textbf{Avg Distance}: average distance traveled per arc (km); 
\textbf{Max Distance}: maximum distance between any two points (km).}
\resizebox{\textwidth}{!}{%
\begin{tabular}{cccccccccccc}
\hline
\textbf{Instance} & \textbf{Dates} & \textbf{Served Mills} & \textbf{Used Blocks} & \textbf{Products} & \textbf{Vehicles} & \textbf{Home bases} & \textbf{Nodes} & \textbf{Arcs} & \textbf{Total Demand (GMT)} & \textbf{Avg Distance} & \textbf{Max Distance} \\ \hline

\multicolumn{12}{c}{\textbf{Small Instances (Arcs $<$ 30,000)}} \\ \hline
$W_1$  & 2025-01-10 -- 2025-01-14 & 11 & 25 & 3 & 70 & 27 & 1295 & 26372 & 6140  & 242.06 & 574.00 \\ \hline
$W_2$  & 2024-12-30 -- 2025-01-04 & 13 & 24 & 3 & 60 & 19 & 1303 & 26504 & 8820  & 272.00 & 629.27 \\ \hline
$W_3$ & 2024-12-24 -- 2024-12-29 & 11 & 19 & 3 & 36 & 19 & 1063 & 17189 & 3820  & 259.31 & 651.28 \\ \hline
$W_4$  & 2024-11-29 -- 2024-12-03 & 13 & 23 & 3 & 51 & 25 & 1267 & 26934 & 7180  & 245.39 & 680.30 \\ \hline
$W_{5}$ & 2024-11-24 -- 2024-11-28 & 15 & 23 & 3 & 53 & 17 & 1321 & 29265 & 10440 & 206.22 & 628.94 \\ \hline

\multicolumn{12}{c}{\textbf{Medium Instances (30,000 $\leq$ Arcs $\leq$ 50,000)}} \\ \hline
$W_6$ & 2024-12-04 -- 2024-12-08 & 14 & 24 & 3 & 57 & 25 & 1333 & 32603 & 9600  & 231.08 & 660.30 \\ \hline
$W_7$  & 2024-12-09 -- 2024-12-13 & 17 & 30 & 3 & 66 & 27 & 1629 & 45237 & 14250 & 227.31 & 574.37 \\ \hline
$W_8$ & 2024-11-19 -- 2024-11-23 & 14 & 22 & 3 & 67 & 27 & 1267 & 32044 & 10800 & 283.70 & 657.06 \\ \hline
$W_{9}$ & 2024-11-14 -- 2024-11-18 & 17 & 23 & 3 & 68 & 30 & 1376 & 37120 & 10320 & 229.56 & 660.30 \\ \hline
$W_{10}$ & 2024-11-09 -- 2024-11-13 & 16 & 21 & 3 & 62 & 26 & 1280 & 31172 & 9630  & 225.54 & 657.06 \\ \hline
$W_{11}$ & 2024-10-30 -- 2024-11-03 & 18 & 46 & 4 & 59 & 27 & 1465 & 44237 & 12930 & 241.16 & 728.76 \\ \hline
$W_{12}$ & 2024-10-14 -- 2024-10-18 & 16 & 26 & 4 & 56 & 26 & 1464 & 36989 & 8140  & 264.96 & 739.14 \\ \hline
$W_{13}$ & 2024-10-10 -- 2024-10-14 & 13 & 37 & 4 & 62 & 27 & 1777 & 40795 & 11330 & 252.04 & 739.14 \\ \hline
$W_{14}$  & 2025-01-05 -- 2025-01-09 & 16 & 31 & 3 & 88 & 27 & 1637 & 43155 & 20130 & 244.92 & 574.00 \\ \hline
\multicolumn{12}{c}{\textbf{Large Instances (Arcs $>$ 50,000)}} \\ \hline
$W_{15}$ & 2024-10-25 -- 2024-10-29 & 17 & 37 & 4 & 72 & 30 & 1888 & 53103 & 12850 & 246.67 & 657.06 \\ \hline
$W_{16}$ & 2024-12-19 -- 2024-12-23 & 17 & 35 & 3 & 71 & 28 & 1808 & 50238 & 10540 & 268.47 & 651.28 \\ \hline
$W_{17}$  & 2024-12-14 -- 2024-12-18 & 17 & 38 & 3 & 64 & 27 & 1917 & 52736 & 11630 & 257.75 & 659.57 \\ \hline
$W_{18}$ & 2024-11-04 -- 2024-11-08 & 18 & 46 & 4 & 72 & 27 & 2233 & 86959 & 20240 & 220.54 & 702.36 \\ \hline
$W_{19}$ & 2024-10-20 -- 2024-10-24 & 17 & 40 & 4 & 78 & 30 & 1996 & 60449 & 23500 & 234.08 & 721.61 \\ \hline
$W_{20}$ & 2024-10-15 -- 2024-10-19 & 16 & 36 & 4 & 85 & 27 & 1825 & 50790 & 17170 & 227.70 & 739.14 \\ \hline

\end{tabular}%
}
\label{instances-characteristics}
\end{table}

\subsection{Fleet characteristics and transportation costs}
The heterogeneous fleet includes eight different truck configurations, 
varying by capacity and whether they are self-loading. Seasonal 
restrictions further influence effective payload capacity. Contractors 
operate between 1 and 30 trucks each, resulting in diverse availability 
levels across the network. The majority of contractors own a single 
truck. Overall, the fleet comprises 299 trucks operated by 59 
contractors, with a weekly average of 64 active trucks and 26 
contractors. The fleet includes 133 self-loader trucks, with an 
average of 29 per week. Eight truck configurations are available 
overall, with an average of 6 configurations active per week.

As provided by our industrial partner, Table~\ref{tab:haul_stop_costs} 
reports the estimated hauling and stop costs (in \$/hour) for each 
truck configuration across different seasons, highlighting the seasonal 
variability in fleet operating costs.

\begin{table}[H]
\centering
\scriptsize
\setlength{\tabcolsep}{4pt}
\renewcommand{\arraystretch}{1.2}
\begin{tabular}{l|cc|cc|cc|cc|cc}
\hline
 & \multicolumn{2}{c|}{Winter 23--24} 
 & \multicolumn{2}{c|}{Summer 23} 
 & \multicolumn{2}{c|}{Winter 22--23} 
 & \multicolumn{2}{c|}{Summer 22} 
 & \multicolumn{2}{c}{Winter 23--24} \\
\cline{2-11}
Configuration & Hauling & Stop & Hauling & Stop & Hauling & Stop 
              & Hauling & Stop & Hauling & Stop \\
 & (\$/Hr) & (\$/Hr) & (\$/Hr) & (\$/Hr) & (\$/Hr) & (\$/Hr) 
 & (\$/Hr) & (\$/Hr) & (\$/Hr) & (\$/Hr) \\
\hline
QUADFLEET  & 107.23 & 96.67 & 104.01 & 93.77 & 101.87 & 91.84 
           & 98.81 & 89.08 & 110.40 & 99.84 \\
TRIFLEET   & 107.23 & 96.67 & 104.01 & 93.77 & 101.87 & 91.84 
           & 98.81 & 89.08 & 110.40 & 97.29 \\
QUADSELF   & 112.69 & 102.13 & 109.31 & 99.07 & 107.06 & 97.02 
           & 103.84 & 94.11 & 116.67 & 106.11 \\
TRISELF    & 112.69 & 102.13 & 109.31 & 99.07 & 107.06 & 97.02 
           & 103.84 & 94.11 & 116.67 & 103.72 \\
BTRAIN     & 113.24 & 97.84 & 109.84 & 94.90 & 107.58 & 92.95 
           & 104.35 & 90.16 & 116.75 & 101.35 \\
BTRAINSELF & 115.52 & 100.12 & 112.05 & 97.12 & 109.74 & 95.11 
           & 106.45 & 92.26 & 120.61 & 105.21 \\
TRD        & 115.46 & 102.15 & 112.00 & 99.09 & 109.69 & 97.04 
           & 106.40 & 93.13 & 118.09 & 104.74 \\
TRS        & 118.37 & 106.05 & 114.82 & 102.87 & 112.45 & 100.75 
           & 109.08 & 97.73 & 121.30 & 108.98 \\
\hline
\end{tabular}
\caption{Hauling and waiting costs (\$/Hr) by fleet configuration 
across different seasons.}
\label{tab:haul_stop_costs}
\end{table}
In our dataset, the travel speed between each forest block and mill is defined at the level of each pair $(f,m)$ and varies across seasons. 
This information allows us to compute the effective travel time $t_{fm} = d_{fm} / v_{fm}$, where $d_{fm}$ is the distance (km) and $v_{fm}$ is the observed average speed (km/h).  
To better characterize the network, we classify each connection into three categories:  
\emph{slow roads} ($v_{fm} < 60$ km/h),  
\emph{medium roads} ($60 \leq v_{fm} < 80$ km/h), and  \emph{fast roads} ($v_{fm} \geq 80$ km/h). 

\subsection{Test plan}
As described previously, we consider three categories of $\mathcal{LTRSP}$ weekly instances $\mathcal{W}$. Based on these categories, we compare the performance of the following methods:

\begin{itemize}
    \item \textbf{$\mathcal{P}_m$-$\mathcal{MILP}_d$}: 
    Main model solved by CPLEX with default parameter settings solved to a MILP gap $\varepsilon = 1\%$.
    
    \item \textbf{$\mathcal{P}_m$-$\mathcal{MILP}_t$}: 
    Main model solved by CPLEX with tuned parameters solved to a MILP gap $\varepsilon = 1\%$.
    
    \item \textbf{$\mathcal{P}_m$-$\mathcal{RF-FO}_d$}: 
    Hybrid metaheuristic (RF-FO) for the main model with default parameters.
    
    \item \textbf{$\mathcal{P}_m$-$\mathcal{RF-FO}_t$}: 
    Hybrid metaheuristic (RF-FO) for the main model with tuned parameters.
\end{itemize}

The performance of each method is assessed using two well-established criteria:  
(i) the runtime $T$ (in minutes), and  
(ii) the optimality gap relative to the \emph{best known objective value} $\mathcal{GAP}_{\text{BS}}$ (in \%).  

The best known objective for each instance $W_i$ is defined as
\[
z^{\star}(W_i) \;=\; \min \{\, z_m(W_i) \;\mid\; m \in 
\{\mathcal{P}_m\mbox{-}\mathcal{MILP}_d,\;
   \mathcal{P}_m\mbox{-}\mathcal{MILP}_t,\;
   \mathcal{P}_m\mbox{-}\mathcal{RF\mbox{-}FO}_d,\;
   \mathcal{P}_m\mbox{-}\mathcal{RF\mbox{-}FO}_t\} \,\},
\]
where $z_m(W_i)$ denotes the objective value obtained by method $m$ for instance $W_i$.  

The relative gap of method $m$ on instance $W_i$ is then given by
\[
\mathcal{GAP}_{\text{BS}}(m,W_i) \;=\;
\frac{z_m(W_i) - z^{\star}(W_i)}{z^{\star}(W_i)} \times 100.
\]

All methods rely on the same model $\mathcal{P}_m$, incorporating the constraints described in Section~\ref{prob-formulation}. The routing network for each weekly instance $W_i$ is generated using modular preprocessing functions that systematically eliminate infeasible arcs. The entire framework was implemented in Python.  

All computational experiments were conducted on a workstation equipped with a 3.20~GHz Intel\textsuperscript{\textregistered} Core\texttrademark{} i7-8700 processor, 64~GB of system memory, and running a Linux operating system. The Integer Linear Programs (ILPs) were solved using the IBM ILOG CPLEX Optimizer (version~22.1.1.0), with the PuLP library (version~3.3.0) used as the modeling interface.

\subsection{$\mathcal{LTRSP}$ numerical results} \label{plann-weeks}
We compare various methods by reporting the execution
time $T$ in minutes and the gap described above
$\mathcal{GAP}_{\text{BS}}$. Table ~\ref{tab:perf_by_cat} reports results of different instances.
\begin{sidewaystable*}[p]
\centering
\caption{Performance comparison of different methods across all instances. The values reported in the Objective columns are scaled by a factor of $10^{2}$ for readability.
}
\label{tab:perf_by_cat}
\scriptsize
\begin{tabular}{l|ccc|ccc|ccc|ccc|c}
\toprule
\multirow{2}{*}{Instance}
& \multicolumn{3}{c|}{$\mathcal{P}_m$-$\mathcal{MILP}_d$}
& \multicolumn{3}{c|}{$\mathcal{P}_m$-$\mathcal{MILP}_t$}
& \multicolumn{3}{c|}{$\mathcal{P}_m$-$\mathcal{RF\mbox{-}FO}_d$}
& \multicolumn{3}{c|}{$\mathcal{P}_m$-$\mathcal{RF\mbox{-}FO}_t$}
& \multirow{2}{*}{Best Known Obj.} \\
\cmidrule(lr){2-4}\cmidrule(lr){5-7}\cmidrule(lr){8-10}\cmidrule(lr){11-13}
& T (min) & Objective  & $\mathcal{GAP}_{\text{BS}}$ (\%)
& T (min) & Objective & $\mathcal{GAP}_{\text{BS}}$ (\%)
& T (min) & Objective & $\mathcal{GAP}_{\text{BS}}$ (\%)
& T (min) & Objective & $\mathcal{GAP}_{\text{BS}}$ (\%)
&  \\
\midrule
\multicolumn{14}{c}{\textbf{Small Instances} ($|A| < 30{,}000$)} \\
$W_1$  & 13.99 & 24023 & 0.00 & 8.60 & 24197 & 0.72 & 7.84 & 24035 & 0.05 & 9.25 & 24089 & 0.27 & 24023 \\
$W_2$  & 7.78  & 33752 & 0.00 & 6.21 & 34068 & 0.94 & 5.28 & 34188 & 1.29 & 6.04 & 33951 & 0.59 & 33752 \\
$W_3$  & 13.32 & 15148 & 1.55 & 5.94 & 14917 & 0.00 & 2.10 & 15009 & 0.62 & 3.82 & 15033 & 0.78 & 14917 \\
$W_4$  & 8.15  & 30002 & 0.67 & 7.21 & 29804 & 0.00 & 4.00 & 29890 & 0.29 & 5.62 & 30018 & 0.72 & 29804 \\
$W_5$  & 20.53 & 40426 & 0.49 & 8.96 & 40740 & 1.27 & 5.95 & 40230 & 0.00 & 5.65 & 40318 & 0.22 & 40230 \\
\midrule
\textbf{Avg} & \textbf{12.75} & -- & \textbf{0.54} & \textbf{7.78} & -- & \textbf{0.58} & \textbf{5.03} & -- & \textbf{0.45} & \textbf{6.08} & -- & \textbf{0.52} & -- \\
\midrule
\multicolumn{14}{c}{\textbf{Medium Instances} ($30{,}000 \le |A| \le 50{,}000$)} \\
$W_6$   & 15.04 & 39966 & 0.45 & 10.88 & 40089 & 0.76 & 6.22 & 39805 & 0.05 & 9.57 & 39786 & 0.00 & 39786 \\
$W_7$   & 38.52 & 55708 & 0.43 & 23.15 & 55674 & 0.37 & 15.85 & 55573 & 0.38 & 16.38 & 55467 & 0.00 & 55467 \\
$W_8$   & 28.97 & 42232 & 0.08 & 11.08 & 42304 & 0.25 & 8.93 & 42198 & 0.00 & 9.50 & 42198 & 0.00 & 42198 \\
$W_9$   & 25.78 & 41453 & 0.44 & 15.83 & 41629 & 0.86 & 9.54 & 41287 & 0.03 & 12.01 & 41273 & 0.00 & 41273 \\
$W_{10}$& 14.89 & 37709 & 0.00 & 11.67 & 37844 & 0.36 & 6.80 & 37727 & 0.05 & 7.15 & 37863 & 0.41 & 37709 \\
$W_{11}$& 18.85 & 46636 & 0.12 & 15.31 & 46743 & 0.35 & 15.62 & 46594 & 0.03 & 11.82 & 46581 & 0.00 & 46581 \\
$W_{12}$& 31.12 & 41746 & 1.66 & 22.10 & 41700 & 1.55 & 11.03 & 41063 & 0.00 & 12.60 & 41099 & 0.09 & 41063 \\
$W_{13}$& 13.44 & 32321 & 0.00 & 8.35  & 32344 & 0.07 & 7.26 & 32288 & 0.00 & 9.16 & 32299 & 0.03 & 32288 \\
$W_{14}$& 20.60 & 80175 & 1.51 & 13.30 & 79320 & 0.42 & 13.47 & 79341 & 0.45 & 19.19 & 78984 & 0.00 & 78984 \\
\midrule
\textbf{Avg} & \textbf{23.25} & -- & \textbf{0.63} & \textbf{14.63} & -- & \textbf{0.44} & \textbf{10.75} & -- & \textbf{0.21} & \textbf{11.71} & -- & \textbf{0.17} & -- \\
\midrule
\multicolumn{14}{c}{\textbf{Large Instances} ($|A| > 50{,}000$)} \\
$W_{15}$& 38.63 & 49386 & 0.14 & 41.51 & 49330 & 0.02 & 18.56 & 49315 & 0.00 & 18.04 & 49337 & 0.04 & 49315 \\
$W_{16}$& 35.71 & 42307 & 0.13 & 42.00 & 42568 & 0.74 & 15.50 & 42295 & 0.10 & 17.65 & 42254 & 0.00 & 42254 \\
$W_{17}$& 56.80 & 48445 & 1.85 & 44.54 & 48087 & 1.11 & 17.91 & 47562 & 0.00 & 20.73 & 47566 & 0.01 & 47562 \\
$W_{18}$& 28.36 & 77351 & 1.11 & 31.03 & 76989 & 0.63 & 15.25 & 76504 & 0.00 & 25.55 & 76713 & 0.27 & 76504 \\
$W_{19}$& 78.59 & 89172 & 0.16 & 46.47 & 89030 & 0.00 & 18.95 & 89392 & 0.41 & 21.80 & 89402 & 0.42 & 89030 \\
$W_{20}$& 37.13 & 62388 & 0.25 & 36.74 & 62780 & 0.87 & 16.78 & 62435 & 0.31 & 18.84 & 62232 & 0.00 & 62232 \\
\midrule
\textbf{Avg} & \textbf{45.87} & -- & \textbf{0.79} & \textbf{40.38} & -- & \textbf{0.56} & \textbf{16.11} & -- & \textbf{0.14} & \textbf{20.77} & -- & \textbf{0.12} & -- \\
\bottomrule
\end{tabular}
\end{sidewaystable*}

These results highlight several important performance aspects across the different instance categories. The exact approaches $\mathcal{P}_m$-$\mathcal{MILP}_d$ and $\mathcal{P}_m$-$\mathcal{MILP}_t$, which enforce a stopping gap of $\varepsilon = 1\%$, provide reliable reference solutions across all instances. The choice of this tight optimality gap is motivated by extensive discussions with our industrial partner, for whom even marginal cost improvements are highly valuable due to the narrow profit margins in forestry transportation operations. In practice, any additional cost reduction—down to a few cents per unit—is considered meaningful. However, further tightening the gap below $1\%$ was deemed impractical, as it would lead large computational times caused by the well-known tailing-off effect in branch-and-bound algorithms. Consistently with this observation, the computational effort remains significant: the average runtime $T$ rises from $12.75$ minutes for small instances to more than $45.87$ minutes for large ones. These values confirm the intrinsic difficulty of solving large-scale models, even when the termination criterion is relaxed to $\varepsilon = 1\%$. Solver tuning in $\mathcal{P}_m$-$\mathcal{MILP}_t$ slightly improves performance, reducing the average runtime from $12.75$ to $7.78$ minutes on small instances and from $45.87$ to $40.38$ minutes on large ones without degrading solution quality. The average gaps $\mathcal{GAP}_{\text{BS}}$ are very close between tuned and default MILP solving approach ($0.58\%$ vs.\ $0.54\%$ on small instances, and $0.56\%$ vs.\ $0.79\%$ on large ones). These observations highlight the fact that even parameters calibration was not sufficient to significantly improve the execution time.

The hybrid matheuristic approaches achieve the best overall balance between runtime and solution quality. The default version $\mathcal{P}_m$-$\mathcal{RF\mbox{-}FO}_d$ consistently identifies solutions that are extremely close to the best known objective, with average gaps of $0.45\%$ (small), $0.21\%$ (medium), and only $0.14\%$ (large). At the same time, its runtimes performance is: $5.03$ minutes on small, $10.75$ on medium, and $16.11$ on large instances. These results confirm that the relax-and-fix decomposition significantly accelerates convergence towards high-quality solutions, especially for large-scale instances where the runtime reduction is substantial.
The tuned hybrid variant $\mathcal{P}_m$-$\mathcal{RF\mbox{-}FO}_t$ yields comparable solution quality, with average gaps ranging from $0.52\%$ on small to $0.12\%$ on large instances. However, this comes at the price of a slightly higher runtimes: $6.1$ minutes on small, $11.7$ minutes on medium, and $20.8$ minutes on large instances. This additional computational effort can be attributed to the iterative refinement loops introduced by the tuned configuration. Unlike in the exact approaches, where tuning proved beneficial due to the large number of integer variables that profit from advanced branching and search parameterization, the same mechanism is less efficient in the metaheuristic framework. In $\mathcal{P}_m$-$\mathcal{RF\mbox{-}FO}_t$, many variables are relaxed during decomposition, which limits the advantages of tuning; instead, the refinement overhead becomes the dominant factor, turning tuning into a computational burden rather than an improvement.

Beyond solution quality, reductions in computational time are also of practical importance, even when they are on the order of a few minutes. In the operational context of our industrial partner, weekly routing and scheduling plans are subject to significant uncertainty and frequent last-minute changes, including demand updates, equipment availability, and operational disruptions. As a result, the planning process may need to be executed multiple times within a limited decision window. In such time-constrained settings, shorter execution times are strongly preferred, as they increase the responsiveness of the decision-support system and facilitate rapid re-optimization when new information becomes available.

Overall, the results confirm that the proposed hybrid mataheuristic framework is the most effective approach in practice: it consistently achieves near-optimal solutions in very competitive runtimes, with gaps to the best known objective almost always below $1\%$. The robustness and scalability of the proposed framework are thus demonstrated, as all real-world problem instances can be solved to near-optimality within practical computing times.

Figure~\ref{fig:runtime_decomposition} analyzes the average runtime decomposition of 
$\mathcal{P}_m$-$\mathcal{MILP}_d$ by instance category. Three phases are identified: 
(i) the search for the first feasible incumbent, (ii) the tailing-off phase after the 
integrality gap reaches $5\%$, and (iii) the remaining branching operations. The 
tailing-off effect dominates the runtime, accounting for more than half of the total 
time across all categories and increasing with instance size (from about $9$ minutes 
on small instances to over $23$ minutes on large ones). The time required to find the 
first incumbent also scales with problem size, representing roughly one quarter of the 
total runtime ($4$, $7$, and more than $10$ minutes for small, medium, and large 
instances, respectively), which reflects the difficulty of generating high-quality 
feasible solutions early in the search. In contrast, the contribution of other branching 
operations remains limited (around $10$--$15\%$) and grows smoothly with instance size. 
Overall, these observations indicate that the tailing-off effect remains the dominant 
bottleneck for large-scale instances, motivating the use of hybrid and matheuristic 
strategies. Figure~\ref{fig:solver_profile_instance3} further illustrates this behavior 
on an instance, where the best bound improves gradually while the incumbent 
decreases in stepwise jumps, eventually reaching a gap of about $1\%$.

\begin{figure}[H]
\centering
\begin{minipage}[t]{0.48\textwidth}
\centering
\begin{tikzpicture}
\begin{axis}[
    ybar stacked,
    bar width=12pt,
    width=\linewidth,
    height=6.5cm,
    ymin=0,
    ymax=60,
    ylabel={Runtime (min)},
    symbolic x coords={Small, Medium, Large},
    xtick=data,
    enlarge x limits=0.25,
    ylabel style={font=\small},
    tick label style={font=\small},
    axis line style={thick},
    grid=major,
    grid style={dashed,gray!30},
    legend style={at={(0.02,0.98)},anchor=north west,draw=none,fill=none,font=\small}
]

\addplot+[fill=teal!60] coordinates {(Small,    3.8) (Medium,6.7) (Large,12.3)};
\addplot+[fill=orange!70] coordinates {(Small,8.3) (Medium,15.6) (Large,25.1)};
\addplot+[fill=purple!50] coordinates {(Small,1.8) (Medium,2.8) (Large,7.6)};

\legend{\tiny Incumbent search, \tiny Tailing-off, \tiny Other branching}

\end{axis}
\end{tikzpicture}
\caption{Average runtime decomposition across categories.}
\label{fig:runtime_decomposition}
\end{minipage}\hfill
\begin{minipage}[t]{0.48\textwidth}
\centering
\begin{tikzpicture}
\begin{axis}[
    width=\linewidth, height=6.5cm,
    xlabel={Time (s)},
    ylabel={Objective value},
    ymajorgrids, xmajorgrids,
    ymin=62000, ymax=67500, xmin=0, xmax=1300,
    legend style={at={(0.02,0.98)},anchor=north west,draw=none,fill=none,font=\small}
]

\addplot[thick, dashed, black]
coordinates {
  (135,62858.5)
  (374,62951.5)
  (710,63043.8)
  (755,63114.3)
  (761,63243.8)
  (799,63303.3)
  (835,63415.3)
  (875,63482.8)
  (879,63484.6)
  (1200,63484.6)
};

\addplot[thick, blue, mark=*]
coordinates {
  (158,66924)
  (214,65889)
  (455,64780)
  (879,64400)
  (1200,64300)
};

\legend{\tiny Lower bound, \tiny Incumbent}

\end{axis}
\end{tikzpicture}
\caption{Solver profile on one example instance to a gap of 1 \%.}
\label{fig:solver_profile_instance3}
\end{minipage}
\end{figure}

\begin{table}[H]
\centering
\caption{Comparison of operational KPIs across instance categories.}
\label{tab:ops_perf}
\scriptsize
\begin{tabular}{llcccc}
\toprule
Category & Instance & $D^{tot}$ (km) & $W$ (h) & $\Delta$ (GMT) & $|V|$ \\
\midrule
\multicolumn{6}{c}{\textbf{Small Instances} ($|A| < 30{,}000$)} \\
 & $W_1$  & 13995.9 & 66.1 & 20 & 23 \\
 & $W_2$  & 17663.2 & 116.7 & 46 & 38 \\
 & $W_3$  & 8556.7 & 44.6& 18 & 18 \\
 & $W_4$  & 18120.0 & 80.5 & 39 & 37 \\
 & $W_5$  & 23943.2 & 108.8 & 48 & 48 \\
\textbf{Average (Small)} & & 16455.8 & 83.4 & 34.2 & 33 \\
\midrule
\multicolumn{6}{c}{\textbf{Medium Instances} ($30{,}000 \le |A| \le 50{,}000$)} \\
 & $W_6$   & 26971.4 & 93.6 & 35 & 36 \\
 & $W_7$   & 32797.7 & 145.4 & 57 & 48 \\
 & $W_8$   & 27561.2 & 93.8 & 54 & 50 \\
 & $W_9$   & 27232.4 & 89.3 & 48 & 52 \\
 & $W_{10}$& 24506.1 & 81.1 & 40 & 44 \\
 & $W_{11}$& 30996.0 & 106.0 & 39 & 43 \\
 & $W_{12}$& 23854.6 & 98.9 & 43 & 46 \\
 & $W_{13}$& 21327.3 & 78.4 & 43 & 41 \\
 & $W_{14}$& 44605.5 & 274.8 & 62 & 63 \\
\textbf{Average (Medium)} & & 28827.9 & 117.3 & 46.8 & 47.0 \\
\midrule
\multicolumn{6}{c}{\textbf{Large Instances} ($|A| > 50{,}000$)} \\
 & $W_{15}$& 30805.5 & 105.9& 58 & 44 \\
 & $W_{16}$& 25683.3 & 117.3 & 49 & 41 \\
 & $W_{17}$& 31884.3 & 109.5 & 52 & 53 \\
 & $W_{18}$& 47668.4 & 189.3 & 62 & 53 \\
 & $W_{19}$& 50888.8 & 212.4 & 66 & 53 \\
 & $W_{20}$& 35267.6 & 131.9 & 52 & 37 \\
\textbf{Average (Large)} & & 37033.0 & 144.8 & 56.5 & 47 \\
\bottomrule
\end{tabular}
\end{table}

The analysis of operational KPIs reported in Table~\ref{tab:ops_perf} provides further insights into the impact of instance size on resource utilization. Let $D^{\text{tot}}$ denote the total traveled distance (km), $W$ the cumulative waiting time (h), $\Delta$ the unmet demand, and $|V|$ the number of vehicles used.  

As expected, both $D^{\text{tot}}$ and $W$ increase with the instance size. On average, small instances yield $D^{\text{tot}} = 16\,\,456$ km and $W = 83.4$ h, while medium instances require $D^{\text{tot}} = 28\,\,828$ km and $W = 117.3$ h, i.e., an increase of $+75.0\%$ and $+40.6\%$, respectively. For large instances, the averages rise further to $D^{\text{tot}} = 37\,\,033$ km and $W = 144.8$ h, confirming the growth in routing complexity and synchronization requirements when the number of arcs increases. Similarly, the number of vehicles $V$ grows with the problem size, from an average of $33$ in small instances to $47$ for medium and large ones. This increase shows that larger instances require more routes and also spread the workload across more vehicles, mainly due to heterogeneous truck configurations and tighter time-window constraints.

Unmet demand $\Delta (GMT)$ also increases with instance size, averaging $34.2$ for small, $46.8$ for medium, and $56.5$ for large instances. This reflects the growing difficulty of satisfying all requests as networks become denser and capacity constraints tighten.  
It is worth mentioning that we also conducted alternative experiments where $\Delta$ was equal to zero across all instances. This outcome was obtained because the demand satisfaction constraint was enforced with greater-than inequalities and unmet demand was sufficiently penalized in the objective function, i.e.,
\[
\sum_{t \in T}\sum_{f \in F}\sum_{v \in V}\sum_{a \in A_{f}^{mp}} q_v x_{avt} + \delta_{mp} \geq d_{mp}, \,\, \forall \, m \in M, \forall \, p \in P \, \, t \in T
\]
which permits slight oversupply. This assumption, validated with our industrial partner, is operationally acceptable for mills with adequate storage capacity. From a business standpoint, oversupply is preferable to unmet demand, though it comes at the cost of longer total total distance  and higher fleet usage.

\subsection{Evaluation of advanced operational constraints  }
\label{sec:synthetic_constraints}

The results reported earlier are obtained using real operational data and focus
on the most frequent planning configurations encountered by our industrial
partner and formulated in model $\mathcal{P}_m$. However, forestry transportation planning is subject to intermittent but important operational constraints including partial truckloads, contractual priority deliveries, and driver
change requirements. Although these situations do not occur on a daily basis,
they can arise at any time and must be handled without redesigning the planning model. To demonstrate the extensibility of the proposed framework, we
conduct an additional experimental study based on the same real instances but endowed with synthetic scenarios of these operational advanced constraints.

The synthetic constraints used in this study were designed 
based on discussions with our forestry partner to ensure 
that the experiments reflect near-realistic planning 
situations:
\begin{itemize}
    \item \textbf{Partial truckloads:} between 2 and 5 
    occurrences per week.
    \item \textbf{Priority rules:} applied to between 
    10\% and 20\% of mills.
    \item \textbf{Driver change requirement:} concerns 
    between 10\% and 20\% of drivers.
\end{itemize}
\noindent

For each constraint configuration, we generate 10 scenarios for every instance $W_i$ and measure the marginal impact on:
(i) the average routing gap with respect to the lower bound when $\mathcal{RF-FO}_d$ is used for solving, and
(ii) the computational runtime in comparison to the baseline instance without additional constraints.

Let $z_i^{\mathrm{LB}}$ denote the lower bound for instance $W_i$, and $z_{i,s}^{\mathrm{RF\text{-}FO}}$ the objective value obtained for scenario $s \in \{1,\dots,|S|\}$.
The average routing gap over all instances is computed as
\[
\mathrm{Gap}_{\mathrm{synth}}(\%)
=
\frac{100}{|\mathcal{W}|}
\sum_{W_i \in \mathcal{W}}
\left(
\frac{1}{|S|}
\sum_{s=1}^{|S|}
\frac{z_{i,s}^{\mathrm{RF\text{-}FO}} - z_i^{\mathrm{LB}}}{z_i^{\mathrm{LB}}}
\right).
\]

Let $t_i^{\mathrm{base}}$ be the runtime of the baseline instance and $t_{i,s}^{\mathrm{const}}$ the runtime under additional constraints.
The average relative runtime increase (in percentage) over all instances is defined as
\[
\Delta t(\%)
=
\frac{100}{|\mathcal{W}|}
\sum_{W_i \in \mathcal{W}}
\left(
\frac{1}{|S|}
\sum_{s=1}^{|S|}
\frac{t_{i,s}^{\mathrm{const}} - t_i^{\mathrm{base}}}{t_i^{\mathrm{base}}}
\right).
\]

Table~\ref{tab:synthetic_results_details} reports the detailed results of this study. The results quantify the additional economic and computational burden
induced by each operational constraint and by their joint activation.

\begin{table}[H]
\centering
\caption{Impact of advanced operational constraints on solution quality and runtime.}
\label{tab:synthetic_results_details}
\begin{tabular}{lcc}
\toprule
\textbf{Constraint scenario}
& $\boldsymbol{\mathrm{Gap}_{\mathrm{synth}}}$ \textbf{(\%)}
& $\boldsymbol{\Delta t}$ \textbf{(\%)} \\
\midrule
Partial loads
& 2.1 & 26.4 \\
Priority rules
& 1.2 &     15.7 \\
Driver change requirement
& 1.5 & 23.2 \\
All constraints combined
& 3.4 & 49.6 \\
\bottomrule 
\end{tabular}
\end{table}

The results indicate that each constraint type induces only a limited degradation in
solution quality produced by \textit{RF-FO}, with $\mathrm{Gap}_{\mathrm{syth}}$
remaining below 4\% even when all constraints are simultaneously activated. This
confirms that $\mathcal{RF-FO}_d$ effectively preserves near-optimal solutions under
advanced operational settings. Similarly, the additional computational effort remains controlled, with a maximum
average runtime increase of 50\% in the most restrictive configuration. Notably, even
with all advanced operational constraints enabled, the resulting instances can still
be solved in under 30 minutes, which is compatible with practical weekly planning
requirements.

\subsection{Impact of discretization interval length}
\label{sec:discretization_impact}

The choice of the time discretization interval $\Delta$ directly affects how
waiting time is represented in the planning model and, consequently, the
robustness of routing and scheduling decisions when executed under real-world
operational variability. In Section~\ref{queue}, we derived a theoretical
discretization interval $\Delta^\star$ by balancing the expected waiting time
and expected idle time. While this analysis provides theoretical guidance, an
empirical sensitivity assessment is conducted to understand how discretization
interacts with execution uncertainty.

To do this, we conducted a two-stage evaluation. First, for each discretization
interval $\Delta$, the planning model is solved using deterministic service
times equal to their known average. Second, the resulting plans are evaluated
under stochastic execution conditions. Specifically, loading and unloading
times are simulated using independent realizations drawn from a normal
distribution with mean $30$~minutes and standard deviation $5$~minutes. For
each discretization value, $20$ independent execution scenarios are generated,
and the objective function is averaged across scenarios and all instances.

This evaluation focuses on the variation of the objective function under
realistic simulated service times, thereby capturing the true operational cost
of executing a plan designed under a given discretization granularity. The
reported cost variation corresponds to the average relative increase of this
objective with respect to the original objective function. Formally, for a
given discretization interval $\Delta$, the average execution-induced objective
variation is defined as
\begin{equation}
\label{eq:avg_obj_variation}
\Delta Z ^{\mathrm{exec}}(\Delta)
=
\frac{1}{|\mathcal{W}| \cdot |S|}
\sum_{W_i |\in \mathcal{W}}
\sum_{s=1}^{|S|}
\frac{
z^{\mathrm{exec}}_{i,s}(\Delta) - z^{\mathrm{plan}}_i(\Delta)
}{
z^{\mathrm{plan}}_i(\Delta)
}
\times 100,
\end{equation}
where $z^{\mathrm{plan}}_i(\Delta)$ denotes the objective value obtained by the
planning model for instance $W_i$ using discretization interval $\Delta$, and
$z^{\mathrm{exec}}_{i,s}(\Delta)$ is the realized objective value under execution
scenario $s \in \{1,\dots,|S|\}$. The metric is averaged over all instances and execution scenarios.

The results are summarized in Table~\ref{tab:discretization_comparison}. Small
discretization intervals, such as $\Delta = 15$ minutes, lead to a large
increase in the objective function. This behavior is explained by the lack of
temporal slack in the plan, which makes the solution highly sensitive to
variations in loading and unloading times and results in truck backlogs during
execution. As the discretization interval increases, the average objective
variation decreases and reaches its minimum around $\Delta = 45$ minutes. For
larger discretization intervals, the objective variation increases again. In
this case, the model introduces excessive temporal slack, which translates into
idle time and higher waiting costs. Overall, these results highlight the
existence of a trade-off between robustness to execution variability and
temporal discretization. They empirically confirm that intermediate
discretization values provide the most balanced and robust performance under
stochastic service times.

\begin{table}[H]
\centering
\caption{Average execution-induced objective variation for different discretization intervals.}
\label{tab:discretization_comparison}
\begin{tabular}{c c}
\toprule
\textbf{Discretization interval $\Delta$ (min)}
& $\boldsymbol{\Delta Z^{\mathrm{exec}}(\Delta)}$ \textbf{(\%)} \\
\midrule
15  & $+31.20$ \\
30  & $+9.89$  \\
45  & $+3.82$  \\
60  & $+12.03$ \\
75  & $+18.34$ \\
\bottomrule
\end{tabular}
\end{table}

\subsection{Modeling and solver behavior study}
\label{accele-strategy}
In addition to the baseline performance analysis presented earlier, we now
investigate key modeling and solver design choices that significantly
impact solution quality and runtime:
\begin{enumerate}
    \item Symmetry breaking across identical vehicles within each contractor.
    \item Replacing big-$M$ formulations with indicator constraints.
    \item The effect of the penalty coefficient on unmet demand.
\end{enumerate}
The following paragraphs report the experimental results 
for each component, using the instances described previously.

\paragraph{Impact of vehicle symmetry breaking.}
Figure~\ref{fig:sb_runtime_all} reports solver runtimes with and without 
symmetry breaking for each instance. The baseline runtimes correspond to 
$\mathcal{P}_m$-$\mathcal{RF\mbox{-}FO}_d$. Symmetry breaking reduces 
solver runtimes by an average of approximately $27\%$ across the $20$ 
weekly instances. The improvement is consistent across all instance sizes: 
small instances exhibit limited but non-negligible reductions, typically 
around $1$--$2$ minutes per run; medium-sized instances benefit from more 
substantial savings of up to $4$ minutes; and large-scale cases experience 
the most pronounced reductions, often exceeding $5$ minutes per instance.
These results confirm the effectiveness of symmetry breaking as a practical 
mechanism for eliminating redundant branching and accelerating convergence, 
particularly in heterogeneous fleet settings.
Accordingly, symmetry breaking is systematically incorporated into all 
solution approaches preceding $\mathcal{P}_m$-$\mathcal{MILP}_d$, 
$\mathcal{P}_m$-$\mathcal{MILP}_t$, $\mathcal{P}_m$-$\mathcal{RF\mbox{-}FO}_d$, 
and $\mathcal{P}_m$-$\mathcal{RF\mbox{-}FO}_t$.

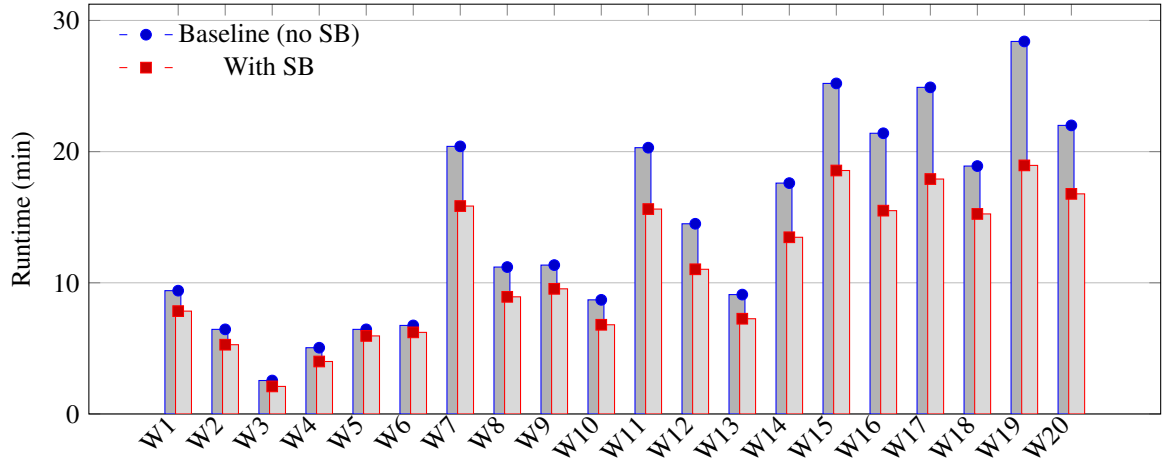
\begin{figure}[H]
\centering
\begin{tikzpicture}
\begin{axis}[
    width=0.95\textwidth,
    height=7cm,
    ymajorgrids,
    ymin=0,
    symbolic x coords={
        W1,W2,W3,W4,W5,W6,W7,W8,W9,W10,
        W11,W12,W13,W14,W15,W16,W17,W18,W19,W20
    },
    xtick=data,
    xticklabel style={rotate=45,anchor=east},
    ylabel={Runtime (min)},
    xlabel={},
    bar width=6pt,
    legend style={
        at={(0.02,0.98)},
        anchor=north west,
        draw=none,
        fill=none
    }
]
\addplot+[ybar,fill=gray!60,bar shift=-2pt] coordinates {
(W1,9.40) (W2,6.45) (W3,2.55) (W4,5.05) (W5,6.45)
(W6,6.75) (W7,20.40) (W8,11.20) (W9,11.35) (W10,8.70)
(W11,20.30) (W12,14.50) (W13,9.10) (W14,17.60) (W15,25.20)
(W16,21.40) (W17,24.90) (W18,18.90) (W19,28.40) (W20,22.00)
};
\addlegendentry{Baseline (no SB)}

\addplot+[ybar,fill=gray!30,bar shift=+2pt] coordinates {
(W1,7.84) (W2,5.28) (W3,2.10) (W4,4.00) (W5,5.95)
(W6,6.22) (W7,15.85) (W8,8.93) (W9,9.54) (W10,6.80)
(W11,15.62) (W12,11.03) (W13,7.26) (W14,13.47) (W15,18.56)
(W16,15.50) (W17,17.91) (W18,15.25) (W19,18.95) (W20,16.78)
};
\addlegendentry{With SB}
\end{axis}
\end{tikzpicture}
\caption{Runtime comparison with and without symmetry breaking for all instances.}
\label{fig:sb_runtime_all}
\end{figure}

\paragraph{Impact of indicator constraints.}
Table~\ref{tab:indicators_cat} summarizes the impact of replacing big-$M$ 
constraints with indicators. Across all categories, runtimes improve by 
about $23\%$ on average. For small instances, the average runtime decreases 
from $7.21$ to $5.03$ minutes, while for medium instances it drops from 
$13.52$ to $10.75$ minutes. For large instances, the reduction is from 
$19.70$ to $16.11$ minutes. These results confirm that indicator 
constraints provide a more faithful relaxation of time windows, improving 
numerical stability and reducing the branch-and-bound effort needed to 
close the integrality gap.

\begin{table}[H]
\centering
\caption{Big-$M$ vs indicator constraints across instance categories.}
\label{tab:indicators_cat}
\small
\begin{tabular}{lcc}
\toprule
Category & Setup & Runtime (min) \\
\midrule
\multirow{2}{*}{Small} 
  & Big-$M$    & 7.21 \\
  & Indicators & 5.03 \\
\midrule
\multirow{2}{*}{Medium} 
  & Big-$M$    & 13.52 \\
  & Indicators & 10.75 \\
\midrule
\multirow{2}{*}{Large} 
  & Big-$M$    & 19.70 \\
  & Indicators & 16.11 \\
\midrule
\textbf{Avg Reduction} 
  & -- & \textbf{23\%} \\
\bottomrule
\end{tabular}
\end{table}

\paragraph{Penalty coefficient on unmet demand.}
The final modeling choice concerns the penalty coefficient $c$ applied to 
total unmet demand $\underline{\Delta_{tot}}$. This term enters the 
objective function as $c \cdot \underline{\Delta_{tot}}$, where 
$\underline{\Delta_{tot}}$ is measured in GMT. Its role is to balance 
routing costs against the cost of shortages: a small $c$ allows the solver 
to neglect some demands in favor of shorter routes, while a large $c$ 
enforces near-complete demand coverage. 
Figure~\ref{fig:penalty_sensitivity} reports the average percentage of 
unmet demand across instance categories when varying 
$c \in \{10,30,50,100,500,1000,10^4\}$.

For very small penalty values ($c=10$), the solver tolerates between 
$2.3\%$ and $2.7\%$ of demand remaining unmet on average. Increasing $c$ 
rapidly reduces this percentage: at $c=50$ the average shortage falls 
below $2\%$, and from $c=100$ onward it stabilizes around 
$1.6\%$--$1.8\%$ across all categories. These results confirm that the 
penalty coefficient effectively prioritizes demand satisfaction, with 
diminishing returns beyond $c=100$. Thus, setting $c$ in the order of 
$10^2$--$10^3$ is sufficient to ensure that nearly all demand is served 
without introducing unnecessary numerical stiffness in the model.

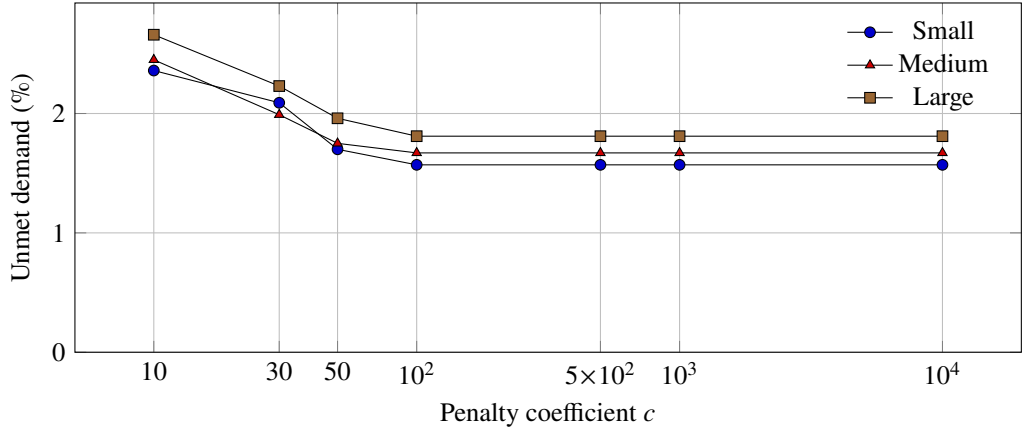
\begin{figure}[H]
\centering
\begin{tikzpicture}
\begin{axis}[
  width=0.85\textwidth,height=6.2cm,
  xlabel={Penalty coefficient $c$}, ylabel={Unmet demand (\%)},
  xmode=log, log basis x={10},
  xtick={10,30,50,100,500,1000,10000},
  xticklabels={$10$,$30$,$50$,$10^2$,$5\!\times\!10^2$,$10^3$,$10^4$},
  ymajorgrids, xmajorgrids,
  legend style={at={(0.98,0.98)},anchor=north east,draw=none,fill=none},
  ymin=0
]
\addplot+[mark=*,black] coordinates {
  (10,2.36) (30,2.09) (50,1.70) (100,1.57) (500,1.57) (1000,1.57) (10000,1.57)};
\addlegendentry{Small}

\addplot+[mark=triangle*,black] coordinates {
  (10,2.45) (30,1.99) (50,1.75) (100,1.67) (500,1.67) (1000,1.67) (10000,1.67)};
\addlegendentry{Medium}

\addplot+[mark=square*,black] coordinates {
  (10,2.66) (30,2.23) (50,1.96) (100,1.81) (500,1.81) (1000,1.81) (10000,1.81)};
\addlegendentry{Large}
\end{axis}
\end{tikzpicture}
\caption{Sensitivity of unmet demand to the penalty coefficient $c$.}
\label{fig:penalty_sensitivity}
\end{figure}
\subsection{Business insights}
A key advantage of the proposed optimization framework to manage the $\mathcal{LTRSP}$ is its ability to produce near-optimal solutions within short computational times. This capability provides our industrial partner with several advantages. From an operational perspective, dispatchers benefit from an operational buffer, i.e., additional time to validate and adjust the proposed plans before implementation. In practice, this buffer increases the robustness of the planning process, as it allows for the anticipation and integration of last-minute adjustments arising from stochastic disruptions such as demand fluctuations, vehicle breakdowns, or road closures.

From an economic standpoint, the solution contributes to a significant reduction in transportation costs, which represent a major component of the total operational expenses. Discussions with our partner revealed that the $\mathcal{LTRSP}$ is currently handled manually, relying mostly on repetitive back-and-forth trips with little optimization and almost no systematic backhauling, except when it is trivial. Furthermore, transportation planning is decentralized, with each contractor managing its own fleet independently. This fragmented approach prevents the exploitation of substantial cost-reduction opportunities. In summary, the proposed optimization framework enhances both the operational robustness and the financial efficiency of our forestry partner.

We conducted a comparative experiment over a period of 50 weeks to quantify the potential
impact of the proposed $\mathcal{LTRSP}$ optimization framework. Specifically, we compared
the transportation plans actually executed by our industrial partner with those generated
by the optimization tool. Table~\ref{comparison-real} summarizes the outcomes of this
comparison
\begin{table}[H]
\centering
\caption{Quantified impacts of the proposed solution across financial, environmental, and operational dimensions.}
\label{comparison-real}
\scriptsize
\begin{tabular}{lcccc}
\toprule
\textbf{Dimension} & \textbf{Metric} & \textbf{Impact on the indutrial partner} & \textbf{Impact on Canada scale}& \textbf{Unit} \\
\midrule
Financial     & Cost savings                        & 8$\times 10^9$ CAD  & 377$\times 10^9$  & Annual transport budget \\
              & Reduction in total traveled distance & 35\%&    ------       & Average (km) \\
\midrule
Environmental & CO$_2$ emission reduction           & 4800 &  521 700        & Tons/year \\
\midrule
Operational   & Average planning time                & $<20$& ------       & Minutes/week \\
\bottomrule
\end{tabular}
\end{table}
As shown in Table~\ref{comparison-real}, the proposed optimization framework consistently generates more efficient transportation plans in both financial and environmental dimensions. The cost savings can be formally expressed as follows:
\[
\Delta C = C_{\text{manual}} - C_{\text{opt}},
\]
where $C_{\text{manual}}$ denotes the observed cost of manually generated plans and $C_{\text{opt}}$ the optimized cost. In practice, $C$ is strongly correlated with the total traveled distance $D$, i.e.,
\[
C \propto D,
\]
so that the relative saving ratio is
\[
\text{Saving ratio} = \frac{D_{\text{manual}} - D_{\text{opt}}}{D_{\text{manual}}} \times 100\%.
\]

The estimation of annual financial savings was performed by the management team of our forestry partner, using the reported reduction in distance and unmet demand. For instance, a $5\%$ reduction of the current transportation budget (approximately $30 \times 10^6$ CAD annually) corresponds to
\[
0.05 \times 30 \times 10^6 = 1.5 \times 10^6 \ \text{CAD/year},
\]
 The environmental impact is estimated using standard truck $CO_2$ emission factors per unit distance, adjusted according to the vehicle load. The impact at the Canada scale is estimated as well based on a proportional extrapolation
of the improvements observed for the industrial partner, using transported volume for all Canadian companies as the
scaling factor. These reductions directly translate into fewer vehicles needed and a lower total operational time. Following discussions with our partner, it was agreed that part of these savings could be reinvested in increasing the remuneration rates of transportation contractors, thereby incentivizing their commitment and facilitating the deployment of the optimization tool in practice.

From a decision-support perspective, rapid solution times enable the systematic evaluation of \emph{what-if} scenarios under multiple operational contingencies. The capacity to explore such scenarios in real time strengthens the resilience of the forestry supply chain, reducing the risk of costly disruptions and enhancing coordination across different stakeholders (mills, contractors, and drivers).

Finally, the optimization framework has been integrated into a web-based decision-support system accessible to dispatchers. This system offers advanced visualization capabilities, including interactive Gantt charts Figure~\ref{fig:gantt_business}, which illustrate a typical daily instance of the $\mathcal{LTRSP}$. The Gantt chart displays the planned routes for each truck as sequences of forest blocks $F_i$ and mills $M_i$, with every trip scheduled within a specific time interval.

\begin{figure}[H]
\centering
\includegraphics[width=1\textwidth]{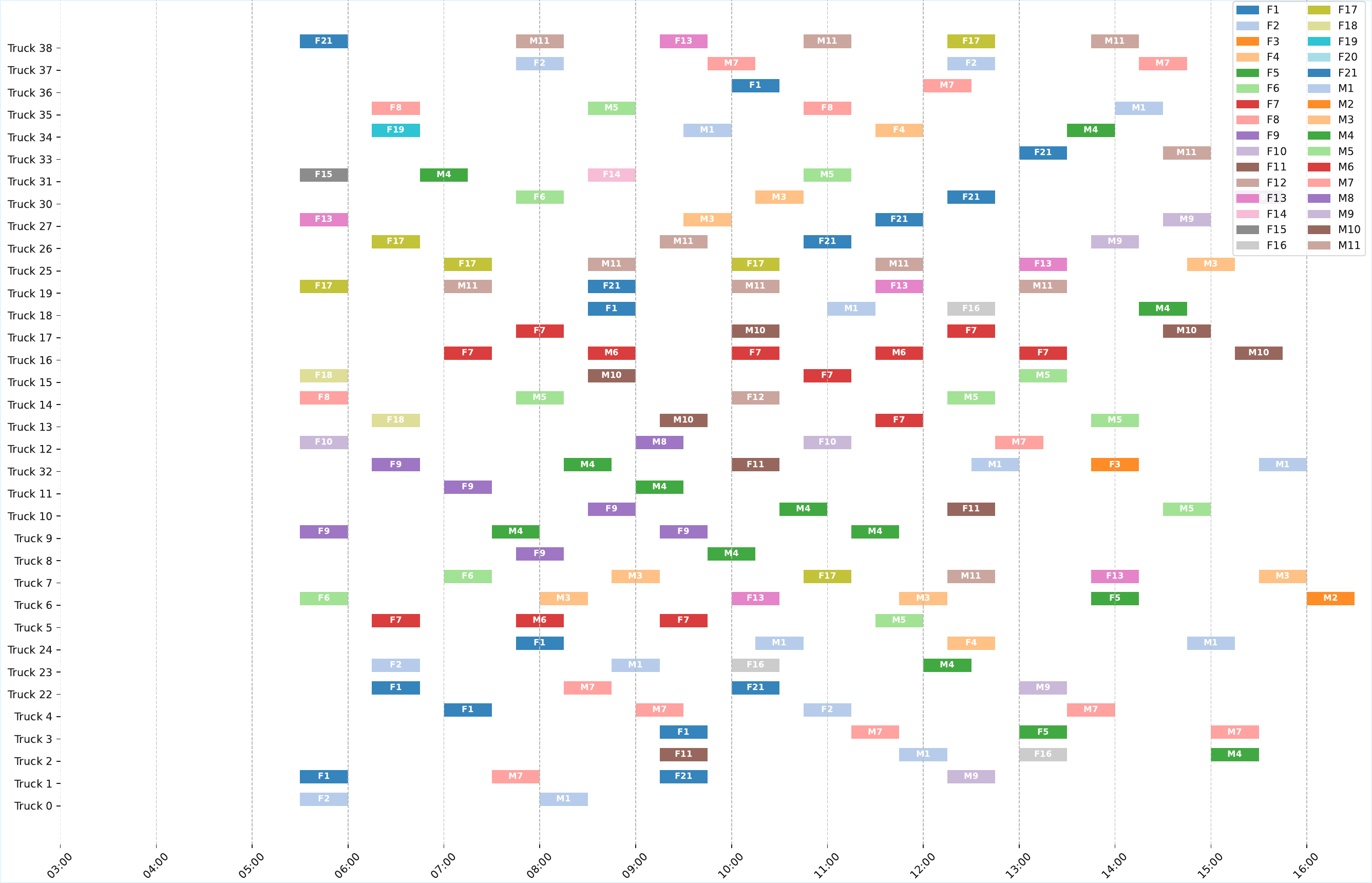} 
\caption{Visualization of daily truck schedules within the web-based application.}
\label{fig:gantt_business}
\end{figure}

\section{Conclusion}

In conclusion, this work provides a comprehensive framework to address the complex multi-period log-truck routing and scheduling problem faced by the Canadian forestry sector. By developing an enhanced mathematical programming formulation and an efficient decomposition strategy, we tackle the challenges posed by the large-scale and multi-period nature of the problem. The proposed framework incorporates the most common business rules inherent to forestry operations, and through a combination of preprocessing strategies, mathematical programming, and advanced solution techniques such as $\mathcal{RF-FO}$, we demonstrate that near-optimal solutions can be obtained efficiently.  

Computational experiments conducted on real-world industrial data underscore the practical relevance of our approach, showing clear improvements in operational efficiency while contributing to environmental sustainability and economic competitiveness of the forest supply chain. Beyond its immediate results, this work is intended to serve as the kernel of a decision support system for transportation planning and routing, providing forest companies and contractors with a robust optimization engine that can be integrated into their daily operations.  

However, some limitations remain. The current formulation assumes deterministic travel times, loading times, and demand levels, whereas in practice these parameters are subject to uncertainty and variability \cite{simard2024improving}. Incorporating stochastic elements into the model—such as probabilistic travel times, loading durations, and demand fluctuations—represents a natural and important extension of this work.  

Finally, as an avenue for future research, we are developing a second layer of real-time reoptimization to dynamically adapt plans in response to disruptions such as route cancelations, delays, or sudden demand changes. Together, these developments will move the framework closer to a fully operational decision support system capable of ensuring robust, efficient, and sustainable log transportation planning in highly uncertain environments.
\section*{Author contributions statement}
Abdelhakim Abdellaoui: Conceptualization, methodology, 
software, formal analysis, investigation, writing original draft, reviewing \& editing.\\
Issmail El Hallaoui: Supervision, conceptualization, 
reviewing, project administration.\\
Loubna Benabbou: Supervision, conceptualization, 
reviewing, project administration.\\
Fran\c{c}ois Aub\'e: Conceptualization, resources, 
funding acquisition, reviewing.\\
Mouloud Amazouz: Resources, funding acquisition, 
project administration.\\
All authors have read and approved the final version of 
the manuscript.

\section*{Acknowledgements}
During the preparation of this manuscript, the authors used 
ChatGPT (GPT-4o) and Claude (Sonnet) to assist with language editing, text reformulation, and manuscript writing. All content generated by these tools was carefully reviewed, revised, and validated by the authors. The authors take full responsibility for the integrity and accuracy 
of the work.

\section*{Data Availability Statement}
The Python code and other data that support the findings 
of this study are available from the corresponding author, 
Abdelhakim Abdellaoui, upon reasonable request.

\section*{Disclosure Statement}
No potential conflict of interest.
\newcommand{\authorphoto}[1]{%
\includegraphics[width=2.5cm, height=3cm, keepaspectratio=false]{#1}%
}

\section*{Notes on contributors}
\begin{minipage}[t]{0.20\textwidth}
\vspace{0pt}
\authorphoto{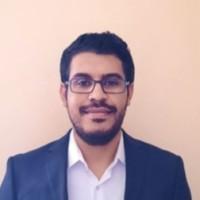}
\end{minipage}
\hfill
\begin{minipage}[t]{0.75\textwidth}
\textbf{Abdelhakim Abdellaoui} is a Ph.D. candidate in the 
Department of Mathematics and Industrial Engineering at 
Polytechnique Montr\'eal, conducting his research at GERAD 
and CanmetENERGY-Varennes, Natural Resources Canada (NRCan). 
His research focuses on mathematical programming, 
decomposition methods, and optimization of large-scale 
routing and scheduling problems, with applications in the 
forestry and transportation sectors. He is also interested 
in the integration of machine learning and deep learning 
techniques within operations research frameworks, 
particularly for combinatorial optimization and 
decision-making under uncertainty.
\end{minipage}
\bigskip
\begin{minipage}[t]{0.20\textwidth}
\vspace{0pt}
\authorphoto{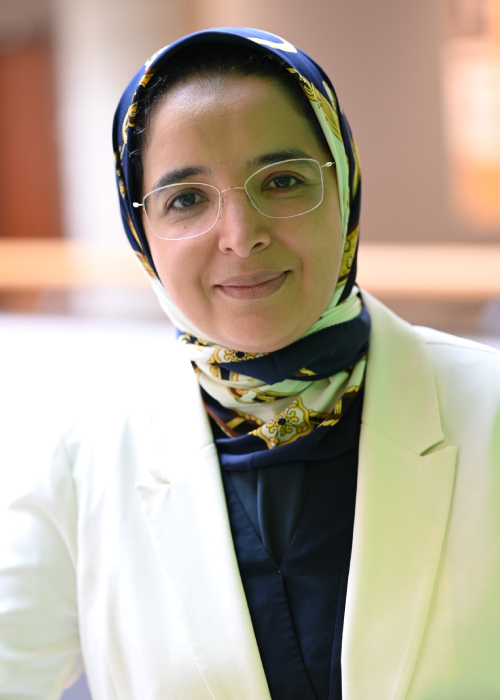}
\end{minipage}
\hfill
\begin{minipage}[t]{0.75\textwidth}
\textbf{Loubna Benabbou} is a Professor at the 
Departement management science, Universit\'e 
du Qu\'ebec \`a Rimouski (UQAR), Campus de L\'evis, 
and a member of GERAD. She holds an M.B.A. and a Ph.D. 
in machine learning and decision sciences from Universit\'e 
Laval, and a degree in industrial engineering from Ecole 
Mohammadia d'Ing\'enieurs. Her research focuses on the 
development and application of machine learning and 
decision science methods for supply chain management, 
industrial process digitalization, and climate change 
risk mitigation.
\end{minipage}
\bigskip
\begin{minipage}[t]{0.18\textwidth}
\vspace{0pt}
\authorphoto{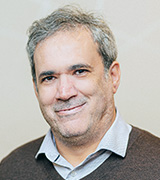}
\end{minipage}
\hfill
\begin{minipage}[t]{0.75\textwidth}
\textbf{Issmail El Hallaoui} is a Full Professor in the 
Department of Mathematics and Industrial Engineering at 
Polytechnique Montr\'eal, and a member of GERAD and IVADO. 
He holds an Ing. degree from ENSIAS (Rabat), and M.Sc. 
and Ph.D. degrees from Polytechnique Montr\'eal. His 
research interests include mathematical programming, 
combinatorial optimization, online optimization and 
approximation algorithms, scheduling and vehicle routing, 
and transportation systems.
\end{minipage}
\bigskip
\begin{minipage}[t]{0.18\textwidth}
\vspace{0pt}
\authorphoto{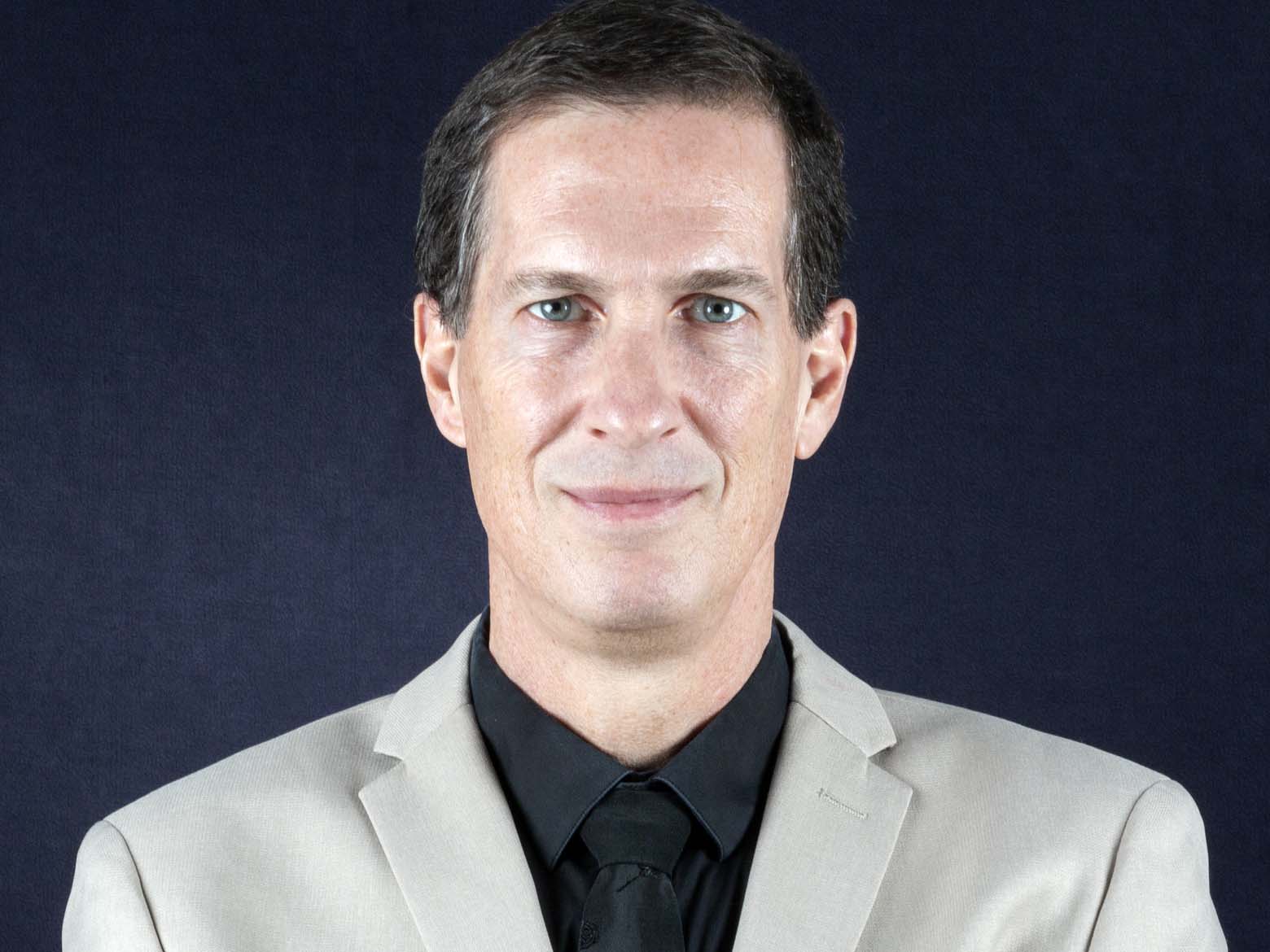}
\end{minipage}
\hfill
\begin{minipage}[t]{0.75\textwidth}
\textbf{Fran\c{c}ois Aub\'e} holds a Ph.D. and is a 
Research Scientist at CanmetENERGY-Varennes, Natural 
Resources Canada (NRCan). His research interests include 
process optimization, machine learning, algorithmic 
development, simulation, and value chain optimization 
in the forestry sector.
\end{minipage}
\bigskip
\begin{minipage}[t]{0.18\textwidth}
\vspace{0pt}
\authorphoto{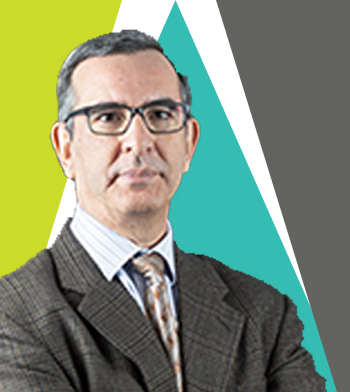}
\end{minipage}
\hfill
\begin{minipage}[t]{0.75\textwidth}
\textbf{Mouloud Amazouz} is a Senior Project Manager at 
CanmetENERGY-Varennes, Natural Resources Canada (NRCan), 
Varennes, Qu\'ebec, Canada. His work focuses on energy 
efficiency, industrial optimization, and the development 
of decision support tools for production and logistics 
systems in the Canadian forestry and energy sectors.
\end{minipage}

\appendix
\section*{ Appendix A. Analytical link between the optimal buffer and the probabilistic buffer.}
\label{appendexA}
We now make explicit the mathematical relationship between the optimal buffer
\( s^\star = \Delta^\star - \overline{L} \) obtained in
Proposition~\ref{prop:frontier} and the probabilistic buffer derived in
Corollary~\ref{cor:cantelli}. This connection is essential to justify both the
necessity of Proposition~\ref{prop:frontier} and the practical relevance of the
corollary.

\subparagraph{Exact expression of the optimal buffer.}
From Proposition~\ref{prop:frontier}, the optimal buffer is given by
\begin{equation}
s^\star
=
\frac{
\frac{\alpha\sigma^2}{2}
+
\sqrt{
\left(\frac{\alpha\sigma^2}{2}\right)^2
+
\alpha\beta\,\overline{L}^{\,2}\sigma^2
}
}{
\beta\,\overline{L}
}.
\label{eq:sstar_exact}
\end{equation}

Factorizing \(\sigma^2\) inside the square root yields
\[
\left(\frac{\alpha\sigma^2}{2}\right)^2
+
\alpha\beta\,\overline{L}^{\,2}\sigma^2
=
\sigma^2\left(
\frac{\alpha^2\sigma^2}{4}
+
\alpha\beta\,\overline{L}^{\,2}
\right).
\]

Taking the square root gives
\[
\sqrt{
\left(\frac{\alpha\sigma^2}{2}\right)^2
+
\alpha\beta\,\overline{L}^{\,2}\sigma^2
}
=
\sigma
\sqrt{
\alpha\beta\,\overline{L}^{\,2}
+
\frac{\alpha^2\sigma^2}{4}
}.
\]

Substituting into \eqref{eq:sstar_exact}, we obtain
\begin{equation}
s^\star
=
\frac{\alpha\sigma^2}{2\beta\overline{L}}
+
\frac{\sigma}{\beta\overline{L}}
\sqrt{
\alpha\beta\,\overline{L}^{\,2}
+
\frac{\alpha^2\sigma^2}{4}
}.
\label{eq:sstar_decomposed}
\end{equation}

\subparagraph{Exact expansion with respect to \(\sigma\).}
Extracting \(\alpha\beta\overline{L}^2\) from the square root in
\eqref{eq:sstar_decomposed} yields
\[
\sqrt{
\alpha\beta\,\overline{L}^{\,2}
+
\frac{\alpha^2\sigma^2}{4}
}
=
\overline{L}\sqrt{\alpha\beta}
\sqrt{
1 + \frac{\alpha\sigma^2}{4\beta\overline{L}^2}
}.
\]

Using the Taylor expansion
\[
\sqrt{1+u} = 1 + \frac{u}{2} + o(u)
\quad \text{as } u \to 0,
\]

We obtain
\[
\sqrt{
1 + \frac{\alpha\sigma^2}{4\beta\overline{L}^2}
}
=
1
+
\frac{\alpha\sigma^2}{8\beta\overline{L}^2}
+
o(\sigma^2).
\]

Substituting the expansion of the square root into
\eqref{eq:sstar_decomposed} yields
\[
s^\star
=
\frac{\alpha\sigma^2}{2\beta\overline{L}}
+
\frac{\sigma}{\beta\overline{L}}
\cdot
\overline{L}\sqrt{\alpha\beta}
\left(
1
+
\frac{\alpha\sigma^2}{8\beta\overline{L}^2}
+
o(\sigma^2)
\right).
\]

We now simplify each term explicitly.
First, observe that
\[
\frac{\sigma}{\beta\overline{L}}
\cdot
\overline{L}\sqrt{\alpha\beta}
=
\sigma\,\frac{\sqrt{\alpha\beta}}{\beta}
=
\sigma\,\sqrt{\frac{\alpha}{\beta}}.
\]

Therefore, the second term becomes
\[
\sigma\,\sqrt{\frac{\alpha}{\beta}}
\left(
1
+
\frac{\alpha\sigma^2}{8\beta\overline{L}^2}
+
o(\sigma^2)
\right).
\]

Expanding the product gives
\[
\sigma\,\sqrt{\frac{\alpha}{\beta}}
+
\sigma\,\sqrt{\frac{\alpha}{\beta}}
\cdot
\frac{\alpha\sigma^2}{8\beta\overline{L}^2}
+
o(\sigma^2).
\]

Since
\[
\sigma\,\sqrt{\frac{\alpha}{\beta}}
\cdot
\frac{\alpha\sigma^2}{8\beta\overline{L}^2}
=
\mathcal{O}(\sigma^3),
\]
this term is absorbed into the \(o(\sigma^2)\) remainder. Collecting all contributions, we obtain
\begin{equation}
s^\star
=
\sigma \sqrt{\frac{\alpha}{\beta}}
+
\frac{\alpha \sigma^2}{2\beta \overline{L}}
+
o(\sigma^2).
\label{eq:sstar_final_expansion}
\end{equation}
In particular,
\begin{equation}
\boxed{s^\star
=
\sigma\sqrt{\frac{\alpha}{\beta}}
+
\mathcal{O}(\sigma^2)}
\quad
\text{as } \sigma \to 0.
\label{eq:sstar_scaling}
\end{equation}
This establishes the exact linear dependence of the optimal buffer on the
service-time standard deviation \(\sigma\).
\section*{ Appendix B. Fractionality of MILP $\mathcal{P}_m$.}
Table~\ref{tab:fractionality-W} reports the fractionality 
rate of routing variables across 20 instances for which the linear relaxation of $\mathcal{P}_m$ is solved. The results indicate that, in all tested cases, the 
proportion of fractional routing variables at optimality remains below $0.5\%$, and is often substantially lower. These empirical values demonstrate that the LP relaxation of the $\mathcal{LTRSP}$
is consistently close to integral across all industrial instances. Only a very small fraction of the binary routing variables require explicit enforcement of integrality, while the vast majority are already fixed at 0 or 1 by the LP relaxation itself.
\begin{table}[ht]
\centering
\caption{Percentage of fractional routing variables in the LP relaxation
for instances $W_1$--$W_{20}$.}
\label{tab:fractionality-W}

\begin{tabular}{c|c c c c c c c c c c}
\hline
Instance & $W_1$ & $W_2$ & $W_3$ & $W_4$ & $W_5$ &
$W_6$ & $W_7$ & $W_8$ & $W_9$ & $W_{10}$ \\
\hline
\textit{Frac(\%)} & 0.42 & 0.37 & 0.18 & 0.25 & 0.31 &
0.40 & 0.21 & 0.16 & 0.33 & 0.29 \\
\hline
\hline
Instance & $W_{11}$ & $W_{12}$ & $W_{13}$ & $W_{14}$ & $W_{15}$ &
$W_{16}$ & $W_{17}$ & $W_{18}$ & $W_{19}$ & $W_{20}$ \\
\hline
\textit{Frac(\%)} & 0.35 & 0.27 & 0.19 & 0.22 & 0.38 &
0.41 & 0.33 & 0.24 & 0.28 & 0.30 \\
\hline
\end{tabular}

\end{table}
\section*{ Appendix C. Illustration of decomposition strategies}
\label{app:decomposition}

We provide a visual illustration of the two 
decomposition strategies employed in this work: the 
\textit{Relax-and-Fix} and the \textit{Fix-and-Optimize} 
heuristics. Both operate by partitioning the planning 
horizon into blocks and iteratively solving sub-problems 
over these blocks.

\subsection*{Relax-and-Fix}

The scheme below illustrates the iterative process of 
the Relax-and-Fix heuristic, which progressively fixes 
decision variables to integer values across planning 
intervals.

\begin{center}
\begin{tikzpicture}
\node at (-1, 2) {First Iteration};
\draw[->][very thick] (0, 3) -- (10, 3) node[right] {Planning Horizon};
\draw[very thick] (0, 2) -- (10, 2);
\foreach \x in {0, 2, 4, 6, 8, 10} \draw[very thick](\x, 1.9) -- (\x, 2.1);
\node at (1, 1.5) {Interval 1};
\node at (3, 1.5) {Interval 2};
\node at (5, 1.5) {Interval 3};
\node at (7, 1.5) {Interval 4};
\node at (9, 1.5) {Interval 5};
\draw[very thick][decorate,decoration={brace,amplitude=10pt,mirror}] 
    (0, 1.3) -- (2, 1.3) node[midway, below=10pt] {Integer Block};
\draw[very thick][decorate,decoration={brace,amplitude=10pt,mirror}] 
    (2.2, 1.3) -- (10, 1.3) node[midway, below=10pt] {Relaxed Block};
\node at (-1, 0) {Second Iteration};
\draw[very thick] (0, 0) -- (10, 0);
\foreach \x in {0, 2, 4, 6, 8, 10} \draw[very thick] (\x, -0.1) -- (\x, 0.1);
\node at (1, -0.5) {Interval 1};
\node at (3, -0.5) {Interval 2};
\node at (5, -0.5) {Interval 3};
\node at (7, -0.5) {Interval 4};
\node at (9, -0.5) {Interval 5};
\draw[very thick][decorate,decoration={brace,amplitude=10pt,mirror}] 
    (0, -0.7) -- (1.9, -0.7) node[midway, below=10pt] {Fixed Block};
\draw[very thick][decorate,decoration={brace,amplitude=10pt,mirror}] 
    (2.1, -0.7) -- (3.9, -0.7) node[midway, below=10pt] {Integer Block};
\draw[very thick][decorate,decoration={brace,amplitude=10pt,mirror}] 
    (4.1, -0.7) -- (10, -0.7) node[midway, below=10pt] {Relaxed Block};
\node at (-1, -2) {Last Iteration};
\draw[very thick] (0, -2) -- (10, -2);
\foreach \x in {0, 2, 4, 6, 8, 10} \draw (\x, -2.1) -- (\x, -1.9);
\node at (1, -2.5) {Interval 1};
\node at (3, -2.5) {Interval 2};
\node at (5, -2.5) {Interval 3};
\node at (7, -2.5) {Interval 4};
\node at (9, -2.5) {Interval 5};
\draw[very thick][decorate,decoration={brace,amplitude=10pt,mirror}] 
    (0, -2.7) -- (7.9, -2.7) node[midway, below=10pt] {Fixed Block};
\draw[very thick][decorate,decoration={brace,amplitude=10pt,mirror}] 
    (8.1, -2.7) -- (10, -2.7) node[midway, below=10pt] {Integer Block};
\end{tikzpicture}
\end{center}

\subsection*{Fix-and-Optimize}

The scheme below illustrates the Fix-and-Optimize heuristic, 
which refines an existing solution by alternately fixing 
and re-optimizing blocks of variables across the planning 
horizon.

\begin{center}
\begin{tikzpicture}
\node at (-1, 2) {First Iteration};
\draw[->][very thick] (0, 3) -- (10, 3) node[right] {Planning Horizon};
\draw[very thick] (0, 2) -- (10, 2);
\foreach \x in {0, 2, 4, 6, 8, 10} \draw[very thick](\x, 1.9) -- (\x, 2.1);
\node at (1, 1.5) {Interval 1};
\node at (3, 1.5) {Interval 2};
\node at (5, 1.5) {Interval 3};
\node at (7, 1.5) {Interval 4};
\node at (9, 1.5) {Interval 5};
\draw[very thick][decorate,decoration={brace,amplitude=10pt,mirror}] 
    (1.1, 1.3) -- (3, 1.3) node[midway, below=10pt] {\,\,\,\,\,Integer Block};
\draw[very thick][decorate,decoration={brace,amplitude=10pt,mirror}] 
    (3.2, 1.3) -- (10, 1.3) node[midway, below=10pt] {Fixed Block};
\draw[very thick][decorate,decoration={brace,amplitude=10pt,mirror}] 
    (0, 1.3) -- (0.9, 1.3) node[midway, below=10pt] {Fixed Block};
\node at (-1, -0.5) {$n$-th Iteration};
\draw[very thick] (0, -0.5) -- (10, -0.5);
\foreach \x in {0, 2, 4, 6, 8, 10} \draw (\x, -0.6) -- (\x, -0.4);
\node at (1, -1) {Interval 1};
\node at (3, -1) {Interval 2};
\node at (5, -1) {Interval 3};
\node at (7, -1) {Interval 4};
\node at (9, -1) {Interval 5};
\draw[very thick][decorate,decoration={brace,amplitude=10pt,mirror}] 
    (0, -1.7) -- (4.8, -1.7) node[midway, below=10pt] {Fixed Block};
\draw[very thick][decorate,decoration={brace,amplitude=10pt,mirror}] 
    (5, -1.7) -- (6.8, -1.7) node[midway, below=10pt] {Integer Block};
\draw[very thick][decorate,decoration={brace,amplitude=10pt,mirror}] 
    (6.9, -1.7) -- (10, -1.7) node[midway, below=10pt] {Fixed Block};
\end{tikzpicture}
\end{center}

\begin{thebibliography}{99}\small
\bibitem{ronnqvist2015operations}Rönnqvist, M., D’Amours, S., Weintraub, A., Jofre, A., Gunn, E., Haight, R., Martell, D., Murray, A. \& Romero, C. Operations research challenges in forestry: 33 open problems. {\em Annals Of Operations Research}. \textbf{232} pp. 11-40 (2015)
\bibitem{NRCan2022} Natural Resources Canada. (2022). The State of Canada's Forests Annual Report 2022.
\bibitem{StatCan2023} Statistics Canada. (2023). Lumber production, shipments and stocks, by Canada and provinces, monthly.
\bibitem{audy2013virtual}Audy, J., D'Amours, S., Favreau, J. \& Rousseau, L. Virtual transportation manager: a decision support system for collaborative forest transportation. (CIRRELT Montreal,2013
\bibitem{audy2023planning}Audy, J., Rönnqvist, M., D’Amours, S. \& Yahiaoui, A. Planning methods and decision support systems in vehicle routing problems for timber transportation: a review. {\em International Journal Of Forest Engineering}. \textbf{34}, 143-167 (2023)
\bibitem{d2008using}D'amours, S., Rönnqvist, M. \& Weintraub, A. Using operational research for supply chain planning in the forest products industry. {\em INFOR: Information Systems And Operational Research}. \textbf{46}, 265-281 (2008)
\bibitem{ronnqvist2003optimization}Rönnqvist, M. Optimization in forestry. {\em Mathematical Programming}. \textbf{97} pp. 267-284 (2003)
\bibitem{Ronnqvist2003} Ronnqvist, M. (2003). Optimization models for forestry. \textit{Forest Science}, 49(3), 345-355.
\bibitem{damavsevivcius2024digital}Damaševičius, R., Mozgeris, G., Kurti, A. \& Maskeliūnas, R. Digital transformation of the future of forestry: an exploration of key concepts in the principles behind Forest 4.0. {\em Frontiers In Forests And Global Change}. \textbf{7} pp. 1424327 (2024)
\bibitem{Flisberg2009}Flisberg, P., Ronnqvist, M. (2009). A decision support system for wood flow problems. \textit{European Journal of Operational Research}, 194(2), 490-502.

\bibitem{Andersson2008} Andersson, K., Ronnqvist, M. (2008). RuttOpt: A decision support system for routing logging trucks. \textit{Operations Research}, 56(4), 1010-1022.

\bibitem{Nadimi2015} Nadimi, S., Ronnqvist, M. (2015). Simulated annealing for the wood chip transportation problem. \textit{Computers and Operations Research}, 62, 1-10.
\bibitem{Forsberg2005} Forsberg, J.,  Ronnqvist, M. (2005). FlowOpt: A decision support tool for strategic and tactical transportation planning in forestry. \textit{Forest Policy and Economics}, 7(5), 765-776.
\bibitem{Ronnqvist2023} Ronnqvist, M. (2023). Hybrid methods for forest transportation distance calculations. \textit{Journal of Forestry Research}, 34(1), 1-15.
\bibitem{audy2023review} Audy, J.-F., Rönnqvist, M., D’Amours, S., Yahiaoui, A.-E. (2023). \textit{Planning methods and decision support systems in vehicle routing problems for timber transportation: A review}. International Journal of Forest Engineering, 34(2), 143–167.
\bibitem{yahiaoui2024mathheuristic} Yahiaoui, A.-E., Rönnqvist, M., Audy, J.-F. (2024). \textit{A Mathheuristic Approach for the Vehicle Routing Problem with Queuing Considerations}. CIRRELT.
\bibitem{weintraub1996simulation} Weintraub, A., Epstein, R., Morales, R., Seron, J., Traverso, P. (1996). \textit{A truck scheduling system improves efficiency in the forest industries}. Interfaces, 26(4), 1–12.
\bibitem{rönnqvist2003models} Rönnqvist, M. (2003). \textit{Optimization models for forest planning}. European Journal of Operational Research, 150(2), 115–130.
\bibitem{palmgren2004column} Palmgren, M., Rönnqvist, M., Varbrand, P. (2004). \textit{A column generation algorithm for a combined timetabling and transportation problem}. European Journal of Operational Research, 156(1), 69–92
\bibitem{shen1978log}Shen, Z. Log truck scheduling by network programming.  (1978)
.

\bibitem{robinson1994tour}Robinson, T. Tour generation for log truck scheduling. {\em 30th Annual Conference Of The Operational Research Society Of New Zealand, August}. pp. 166-171 (1994)
\bibitem{audy2011weekly} Audy, J.-F., D’Amours, S., Rönnqvist, M. (2011). \textit{An operational weekly planning approach for forest transportation with a heterogeneous fleet}. European Journal of Operational Research, 215(2), 329–339.
\bibitem{el_hachemi2013decomposition} El Hachemi, N., Rönnqvist, M., Flisberg, P. (2013). \textit{Decomposition approaches for weekly log truck scheduling}. European Journal of Operational Research, 224(2), 472–480.
\bibitem{ronnqvist1995real} Rönnqvist, M., Ryan, D. M. (1995). \textit{Solving truck dispatching problems in real-time}. European Journal of Operational Research, 83(1), 94–110.
\bibitem{marques2014hybrid} Marques, A., Pinho de Sousa, J.,  Ferreira, C. (2014). \textit{A hybrid approach for tactical and operational forest transportation planning}. European Journal of Operational Research, 236(2), 681–693.
\bibitem{joncour2023generalized}Joncour, C., Kritter, J., Michel, S. \& Schepler, X. Generalized relax-and-fix heuristic. {\em Computers \& Operations Research}. \textbf{149} pp. 106038 (2023)
\bibitem{rey2009column}Rey, P., Muñoz, J. \& Weintraub, A. A column generation model for truck routing in the Chilean forest industry. {\em INFOR: Information Systems And Operational Research}. \textbf{47}, 215-221 (2009)
\bibitem{el2011hybrid}El Hachemi, N., Gendreau, M. \& Rousseau, L. A hybrid constraint programming approach to the log-truck scheduling problem. {\em Annals Of Operations Research}. \textbf{184}, 163-178 (2011)
\bibitem{haridass2014scheduling}Haridass, K., Valenzuela, J., Yucekaya, A. \& McDonald, T. Scheduling a log transport system using simulated annealing. {\em Information Sciences}. \textbf{264} pp. 302-316 (2014)
\bibitem{rix2015column}Rix, G., Rousseau, L. \& Pesant, G. A column generation algorithm for tactical timber transportation planning. {\em Journal Of The Operational Research Society}. \textbf{66}, 278-287 (2015)
\bibitem{asmussen2003applied}Asmussen, S. Applied probability and queues. (Springer,2003)
\bibitem{gil2016log}Gil, A. \& Frayret, J. Log classification in the hardwood timber industry: method and value analysis. {\em International Journal Of Production Research}. \textbf{54}, 4669-4688 (2016)
\bibitem{cantelli1929sui}Cantelli, F. Sui confini della probabilita. {\em Atti Del Congresso Internazionale Dei Matematici: Bologna Del 3 Al 10 De Settembre Di 1928}. pp. 47-60 (1929)
\bibitem{kingman1961single}Kingman, J. The single server queue in heavy traffic. {\em Mathematical Proceedings Of The Cambridge Philosophical Society}. \textbf{57}, 902-904 (1961)
\bibitem{bordon2020mixed}Bordon, M., Montagna, J. \& Corsano, G. Mixed integer linear programming approaches for solving the raw material allocation, routing and scheduling problems in the forest industry. (Growing Science,2020)
\bibitem{acuna2017timber}Acuna, M. Timber and biomass transport optimization: A review of planning issues, solution techniques and decision support tools. {\em Croatian Journal Of Forest Engineering: Journal For Theory And Application Of Forestry Engineering}. \textbf{38}, 279-290 (2017)
\bibitem{melchiori2022mathematical}Melchiori, L., Nasini, G., Montagna, J. \& Corsano, G. A mathematical modeling for simultaneous routing and scheduling of logging trucks in the forest supply chain. {\em Forest Policy And Economics}. \textbf{136} pp. 102693 (2022)
\bibitem{vitale2021optimizing}Vitale, I., Broz, D. \& Dondo, R. Optimizing log transportation in the Argentinean forest industry by column generation. {\em Forest Policy And Economics}. \textbf{128} pp. 102483 (2021)
\bibitem{amrouss2017real}Amrouss, A., El Hachemi, N., Gendreau, M. \& Gendron, B. Real-time management of transportation disruptions in forestry. {\em Computers \& Operations Research}. \textbf{83} pp. 95-105 (2017)
\bibitem{ghotb2024optimization}Ghotb, S., Sowlati, T. \& Mortyn, J. An optimization model for detailed scheduling of heterogeneous fleet of log trucks considering synchronization. {\em International Journal Of Forest Engineering}. \textbf{35}, 313-325 (2024)
\bibitem{wang2023relax}Wang, S., Zhang, H., Chu, F. \& Yu, L. A relax-and-fix method for clothes inventory balancing scheduling problem. {\em International Journal Of Production Research}. \textbf{61}, 7085-7104 (2023)
\bibitem{brahimi2015integrating}Brahimi, N., Aouam, T. \& Aghezzaf, E. Integrating order acceptance decisions with flexible due dates in a production planning model with load-dependent lead times. {\em International Journal Of Production Research}. \textbf{53}, 3810-3822 (2015)
\bibitem{simard2024improving}Simard, V., Rönnqvist, M., LeBel, L. \& Lehoux, N. Improving the decision-making process by considering supply uncertainty–a case study in the forest value chain. {\em International Journal Of Production Research}. \textbf{62}, 665-684 (2024)
\end{thebibliography}
\end{document}